\documentclass[12pt]{article}
\usepackage[a4paper,left=3cm, right=3cm,top=3.5cm, bottom=4.5cm]{geometry}
\usepackage{amsmath,amssymb,amsfonts, marvosym}
\usepackage{epsfig}
\usepackage{tikz}
\usepackage{graphics}
\usepackage{bbm}
\usepackage{bm, bbm}
\usepackage[thmmarks]{ntheorem}

\usetikzlibrary[graphs, arrows, backgrounds, intersections, positioning, fit, petri, calc, shapes, decorations.pathmorphing]
\usetikzlibrary{arrows.meta}
\usepgflibrary{patterns}
\tikzset{node distance=1cm, bend angle=20,
vertex/.style={circle,minimum size=2mm,very thick, draw=black, fill=black, inner sep=0mm}, information text/.style={inner sep=1ex, font=\Large}, help lines/.style={-,color=black, >=stealth', shorten <=.5pt, shorten >=.5pt}, blue help lines/.style={help lines,color=darkblue}, red help lines/.style={help lines, color=darkred}}

\definecolor{darkred}{RGB}{105,0,0}

\makeatletter
\newtheoremstyle{prime}%
 {\item[\hskip\labelsep \theorem@headerfont ##1\ \theorem@separator]}%
{\item[\hskip\labelsep \theorem@headerfont ##1\ ##3' \theorem@separator]}
\makeatother

\makeatletter
\newtheoremstyle{proofof}
{\item[\hskip\labelsep \theorem@headerfont ##1\ \theorem@separator]}%
{\item[\hskip\labelsep \theorem@headerfont ##1\ ##3\theorem@separator]}
\makeatother

\newtheorem{theorem}{Theorem}

\newtheorem{proposition}[theorem]{Proposition}

\newtheorem{corollary}[theorem]{Corollary}

\newtheorem{claim}{Claim}

\newcommand{\ncase}[2]{\smallskip {\bf Case #1\/:} {\it #2}}

\def\pfssp{\hskip 0.5em}
\newcommand{\pffbf}[2]{\noindent{\bf Proof~of\,\,#1\,:} \pfssp #2 \qed \smallskip}

\def \qed {\hfill $\Box$}
\def \QD1 {\hfill $\spadesuit$}

\newcommand{\DF}[1]{{\bf #1\/}}

\newcommand{\set}[2]{\{#1 \;|\; #2 \}}
\newcommand{\ems}{\varnothing}
\newcommand{\sm}{\setminus}

\newcommand{\De}{\Delta}
\newcommand{\de}{\delta}

\newcommand{\cn}{\chi}
\newcommand{\lcn}{\chi_{\ell}}
\newcommand{\dcn}{\chi_{\rm DP}}

\newcommand{\f}{\varphi}
\newcommand{\fin}{\varphi^{-1}}

\newcommand{\cB}{{\cal B}}
\newcommand{\cP}{{\cal P}}

\newcommand{\cG}{{\cal G}}
\newcommand{\cDG}{{\bf Dir}}
\newcommand{\cD}{{\cal D}}
\newcommand{\cO}{{\cal O}}

\newcommand{\CR}{{\rm CR}}

\newcommand{\cK}{\mathfrak{C}}

\newcommand{\nat}{\mathbb{N}}
\newcommand{\nato}{\mathbb{N}_0}
\newcommand{\ganz}{\mathbb{Z}}

\newcommand{\dom}{{\rm dom}}
\newcommand{\su}{{\rm sp}}
\newcommand{\suc}{{\rm sp^o}}

\newcommand{\NP}  {{\sf NP}}
\newcommand{\NPC} {{\NP}-complete}

\newcommand{\mdo}[3]{#1 \equiv #2 \, \mathrm{(mod} \; #3 \mathrm{)}}

\newcommand{\vf}{{\bf f}}

\newcommand{\jou}[4]{{\em #1} {\bf #2} (#3) #4.}
\def \JCTB {J. Combin. Theory \, Ser.~B}

\numberwithin{equation}{section}

\theoremstyle {nonumberplain}
\theoremseparator{:}
\theorembodyfont{\normalfont}
\theoremsymbol{\rule{1ex}{1ex}}
\newtheorem{proof}{Proof}

\theoremstyle{proofof}

\theoremsymbol{$\square$}
\newtheorem{proof2}{Proof}

\theoremsymbol{$\diamond$}

\begin{document}

\title{\bf Generalized DP-Colorings of Graphs}

\author{{{Alexandr~V.~Kostochka}\thanks{Research supported by the Simon Visiting Professor}
\thanks{Department of Mathematics, University of Illinois, Urbana, IL 61801, E-mail address:
kostochk@math.uiuc.edu.}}
\and
{{Thomas Schweser
}\thanks{
Technische Universit\"at Ilmenau, Inst. of Math., PF 100565, D-98684 Ilmenau, Germany. E-mail
address: thomas.schweser@tu-ilmenau.de}}
\and
{{Michael Stiebitz}\thanks{
Technische Universit\"at Ilmenau, Inst. of Math., PF 100565, D-98684 Ilmenau, Germany. E-mail
address: michael.stiebitz@tu-ilmenau.de}}
}

\date{}
\maketitle

\begin{abstract}
By a graph we mean a finite undirected graph having multiple edges but no loops. Given a graph property $\cP$, a $\cP$-coloring of a graph $G$ with color set $C$ is a mapping $\f:V(G)\to C$ such that for each color $c\in C$ the subgraph of $G$ induced by the color class $\fin(c)$ belongs to $\cP$. The $\cP$-chromatic number $\cn(G:\cP)$ of $G$ is the least number $k$ for which $G$ admits an $\cP$-coloring with a set of $k$-colors. This coloring concept dates back to the late 1960s and is commonly known as generalized coloring. In the 1980s the $\cP$-choice number $\lcn(G:\cP)$ of $G$ was introduced and investigated by several authors. In 2018 \v{D}vor\'ak and Postle introduced the DP-chromatic number as a natural extension of the choice number. They also remarked that this concept applies to any graph property. This motivated us to investigate the $\cP$-DP-chromatic number $\dcn(G:\cP)$ of $G$. We have $\cn(G:\cP)\leq \lcn(G:\cP)\leq \dcn(G:\cP)$. In this paper we show that various fundamental coloring results, in particular, the theorems of Brooks, of Gallai, and of Erd\H{o}s, Rubin and Taylor, have counterparts for the $\cP$-DP-chromatic number. Furthermore, we provide a generalization of a result from 2000 about partitions of graphs into a fixed number of induced subgraphs with bounded variable degeneracy due to Borodin, Kostochka, and Toft.
\end{abstract}



\section{Introduction and main results}
\label{sec:introduction}

Our notation is standard. In particular, $\nat$ denotes the set of positive integers and $\nato=\nat \cup \{0\}$. For integers $k$ and $\ell$, let $[k,\ell]=\set{x\in \ganz}{k \leq x \leq \ell}$. The term \DF{graph} refers to a finite undirected graph possibly with multiple edges but without loops.  For a graph $G$, $V(G)$ and $E(G)$ denote the \DF{vertex set} and the \DF{edge set} of $G$, respectively. The number of vertices of $G$ is called the \DF{order} of $G$ and is denoted by $|G|$. A graph $G$ is called \DF{empty} if $|G|=0$; in this case we also write $G=\ems$. For a vertex $v$ of $G$, let $E_G(v)$ denote the set of edges of $G$ \DF{incident} with $e$. Recall that every edge $e$ of $G$ is incident with exactly two vertices of $G$ which are called the \DF{ends} of $e$. We call $d_G(v)=|E_G(v)|$ the \DF{degree} of $v$ in $G$. Then $\De(G)=\max_v d_G(v)$ is the \DF{maximum degree} of $G$, and $\de(G)=\min_v d_G(v)$ is the \DF{minimum degree} of $G$, where we set $\De(\ems)=\de(\ems)=0$. For two different vertices $u, v$ of $G$, let $E_G(u,v)=E_G(u) \cap E_G(v)$. If $e\in E_G(u,v)$, then we also say that $e$ is an edge of $G$ \DF{joining} $u$ and $v$; and that $u$ is a \DF{neighbour} of $v$ and vice versa. Furthermore, $\mu_G(u,v)=|E_G(u,v)|$ is the \DF{multiplicity} of the vertex pair $u,v$ in $G$; and $\mu(G)=\max_{u\not=v}\mu_G(u,v)$ is the \DF{maximum multiplicity} of $G$. The graph $G$ is said to be \DF{simple} if $\mu(G)\leq 1$.
As usual, we denote by $N_G(v)$ the \DF{neighborhood} of $v$ in $G$, that is, the set of vertices $u$ of $G$ with $E_G(u,v)\not=\ems$. A graph $G$ is called \DF{$k$-degenerate} if each subgraph $H$ of $G$ satisfies $\de(H)\leq k$. For $X,Y \subseteq V(G)$, we denote by $E_G(X,Y)$ the set of edges of $G$ joining a vertex of $X$ with a vertex of $Y$. Furthermore, $G[X]$ is the \DF{subgraph} of $G$ \DF{induced} by $X$, i.e., $V(G[X])=X$ and $E(G[X])=E_G(X,X)$. Define $G-X=G[V(G)\sm X]$, and, for $v\in V(G)$, define $G-v=G-\{v\}$. If $G'$ is a \DF{subgraph} of $G$, we write $G'\subseteq G$, that is, $V(G')\subseteq V(G)$, $E(G')\subseteq E(G)$, and each edge of $G'$ has the same ends in $G'$ as in $G$. If $G'\subseteq G$ and $G'\not=G$, then $G'$ is a \DF{proper subgraph} of $G$.
A vertex set $I\subseteq V(G)$ is \DF{independent} in $G$ if $G[I]$ has no edges. A matching of a graph $G$ is a set $M$ of edges of $G$ with no common ends; the matching $M$ is called \DF{perfect} if $|M|=\frac{|G|}{2}$, or equivalently, if every vertex of $G$ is an end of exactly one edge of $M$. A \DF{separating vertex} of a connected graph $G$ is a vertex $v \in V(G)$ such that $G-v$ has at least two components. The separating vertices of a disconnected graph are defined to be those of its components. We denote by $S(G)$ the set of separating vertices of $G$. Furthermore, a \DF{block} of $G$ is a maximal connected subgraph $B$ of $G$ such that $S(B)=\ems$. Note that each block of $G$ is an induced subgraph of $G$. If $\cB(G)=\{G\}$, we also say that $G$ is a block. We denote by $K_n$ the \DF{complete graph} of order $n\geq 0$ and by $C_n$ the \DF{cycle} of order $n\geq 3$. A cycle is said to be \DF{even} or \DF{odd} depending on whether its order is even or odd.
Clearly, both $K_n$ with $n\geq 1$ and $C_n$ with $n\geq 3$ are blocks and simple graphs. For a graph $G$, we denote by $G^o$ the \DF{underlying simple graph} of $G$, that is, $G^o$ is a simple graph with $V(G^o)=V(G)$ and $E(G^o)=\set{uv}{u, v\in V(G), \mu_G(u,v)>0}$. Note that $G$ and $G^o$ have the same block structure, that is, for every $X\subseteq V(G)$ we have $G[X]\in \cB(G)$ if and only if $G^o[X]\in \cB(G^o)$.

Given a graph $G$, a \DF{coloring} of $G$ with \DF{color set} $C$ is a mapping $\f:V(G)\to C$. Then, the sets $\fin(c)=\set{v\in V(G)}{\f(v)=c}$ with $c\in C$ are called \DF{color classes} of the coloring $\f$. A \DF{list assignment} of $G$ with color set $C$ is a mapping $L:V \to 2^C$ that assigns to each vertex $v\in V$ a set (list) $L(v)\subseteq C$ of colors. A coloring $\f$ of $G$ is called an \DF{$L$-coloring} if $\f(v)\in L(v)$ for all $v\in V$. A \DF{cover} of $G$ is a pair $(X,H)$ consisting of a map $X$ and a graph $H$ satisfying the following two conditions:

\begin{description}
\item[(C1)] $X: V \to 2^{V(H)}$ is a function that assigns to each vertex $v \in V$ a vertex set $X_v = X(v) \subseteq V(H)$ such that the sets $X_v$ with $v \in V$ are pairwise disjoint.

\item[(C2)] $H$ is a graph with vertex set $V(H) = \bigcup_{v \in V(G)}X_v$ such that each $X_v$ is an independent set of $H$, and, for any two distinct vertices $u, v \in V(G)$, the set $E_H(X_u,X_v)$ is the union of $\mu_G(u,v)$ (possibly empty) matchings of $H$.
\end{description}

Let $G$ be a graph and let $(X,H)$ be a cover of $G$. Let $uv\in E(G^o)$, let $X\subseteq X_u$, and $Y\subseteq X_v$. Then define $H(X,Y)=H[X\cup Y]$; note that $H(X,Y)$ is a bipartite graph with parts $X$ and $Y$, and $\De(H(X_u,X_v))\leq \mu_G(u,v)$ (by (C2)).  If $|X_v|\geq k$ for all $v\in V(G)$, we say that $(X,H)$ is a \DF{$k$-cover} of $G$. A \DF{transversal} of $(X,H)$ is a vertex set $T\subseteq V(H)$ such that $|T \cap X_v|=1$ for all $v \in V$. A set $T\subseteq V(H)$ is called \DF{partial transversal} of $(X,H)$ if $|T \cap X_v|\leq 1$ for all $v\in V$. For $Y\subseteq V(H)$, let $\dom(Y:G)=\set{v\in V(G)}{X_v\cap Y\not=\ems}$ be the \DF{domain} of $Y$ in $G$.


%

Colorings of graphs become a subject of interest only
when some restrictions to the color classes are imposed. Let $\cG$ denote the class of all graphs. A \DF{graph property} is a subclass of $\cG$ that is closed with respect to isomorphisms. Let $\cP$ be a graph property. The property $\cP$ is said to be \DF{non-trivial} if $\cP$ contains a non-empty graph, but not all graphs. We call $\cP$ \DF{monotone} if $\cP$ is closed under taking subgraphs; and we call $\cP$ \DF{hereditary} if $\cP$ is closed under taking induced subgraphs. If $\cP$ is closed under taking (vertex) disjoint unions, then $\cP$ is called \DF{additive}. Clearly, every monotone graph property is hereditary, but not conversely. An overview about hereditary graph properties is given in \cite{BorowieckiBrFr97}. Some popular graph properties that are non-trivial, monotone, and additive are the following:
$$\cO=\set{G\in \cG}{G \mbox{ is edgeless}},$$
and
$$\cD_k=\set{G\in \cG}{G \mbox{ is $k$-degenerate}}$$
with $k\geq 0$. Note that $\cD_0=\cO$, $\cD_1$ is the class of forests, and $\cO\subseteq \cD_k\subseteq \cD_{k+1}$ for all $k\geq 0$. If $\cP$ is additive, then a graph belongs to $\cP$ if and only if each of its components belong to $\cP$. For a non-trivial and hereditary graph property $\cP$, let
$$\CR(\cP)=\set{G\in \cG}{G\not\in \cP, \mbox{ but } G-v\in \cP \mbox{ for all } v\in V(G)}$$
and define
$$d(\cP)=\min\set{\de(G)}{G\in CR(\cP)}.$$
Note that $\CR(\cD_k)$ contains all connected $(k+1)$-regular graphs and $d(\cD_k)=k+1$. In particular $\CR(\cO)=\langle K_2 \rangle$, that is, each graph in $\CR(\cO)$ is isomorphic to $K_2$,  and $d(\cO)=1$. The statements of the following proposition are well known  and easy to prove (see e.g. \cite[Proposition 1]{Schweser18}).

\begin{proposition} Let $\cP$ be a non-trivial and hereditary graph property. Then the following statements hold:
\begin{itemize}
\item[{\rm (a)}] $K_0, K_1 \in \cP$.
\item[{\rm (b)}] A graph $G$ belongs to $\CR(\cP)$ if and only if each proper induced subgraph of $G$ belongs to $\cP$, but $G$ itself does not belong to $\cP$.
\item[{\rm (c)}] $G\not\in \cP$ if and only if $G$ contains an induced subgraph $G'$ with $G'\in \CR(\cP)$.
\item[{\rm (d)}] $\CR(\cP)\not= \ems$ and $d(\cP)\in \nato$.
\item[{\rm (e)}] If $G\not\in \cP$, but $G-v\in \cP$ for some vertex $v$ of $G$, then $d_G(v)\geq d(\cP)$.
\end{itemize}
\label{prop:smooth}
\end{proposition}

Let $\cP$ be a graph property, and let $G$ be a graph. A \DF{$\cP$-coloring} of $G$ with color set $C$ is a coloring $\f$ of $G$ with color set $C$  such that $G[\fin(c)]\in \cP$ for all $c\in C$. If $L$ is a list assignment for $G$, then a \DF{$(\cP,L)$-coloring} of $G$ is an $\cP$-coloring $\f$ of $G$ such that $\f(v)\in L(v)$ for all $v\in V(G)$. The \DF{$\cP$-chromatic number} of $G$, denoted by $\cn(G:\cP)$, is the least integer $k$ for which $G$ admits a $\cP$-coloring with a set of $k$ colors. The \DF{$\cP$-choice number} of $G$, denoted by $\lcn(G:\cP)$, is the least integer $k$ such that $G$ has an $(\cP,L)$-coloring whenever $L$ is a list assignment of $G$ satisfying $|L(v)|\geq k$ for all $v\in V(G)$. If $(X,H)$ is a cover of $G$, then a \DF{$\cP$-transversal} of $(X,H)$ is a transversal $T$ of $(X,H)$ such that $H[T]\in \cP$. An $\cO$-transversal of $(X,H)$ is also referred to as an \DF{independent transversal} of $(X,H)$. A $\cP$-transversal of $(X,H)$ is also called a \DF{$(\cP,(X,H))$-coloring} of $G$. Note that $G$ admits a $(\cP,(X,H))$-coloring if and only if $G$
has a coloring $\f$ with color set $V(H)$ such that $T=\set{\f(v)}{v \in V(G)}$ is a $\cP$-transversal of $(X,H)$. The \DF{$\cP$-DP-chromatic number} of $G$, denoted by $\dcn(G:\cP)$, is the least integer $k$ such that $G$ admits a $(\cP,(X,H))$-coloring whenever $(X,H)$ is a $k$-cover of $G$. We also write $\cn(G), \lcn(G)$ and $\dcn(G)$ for $\cn(G:\cO), \lcn(G:\cO)$ and $\dcn(G:\cO)$, and the corresponding terms are \DF{chromatic number}, \DF{choice number}, and \DF{\text{DP}-chromatic number}, respectively. The choice number was introduced, independently,  by Vizing \cite{Vizing76} and by Erd\H{o}s, Rubin, and Taylor \cite{ErdosRubTay79}. The DP-chromatic number was introduced by \v{D}vor{\'a}k and Postle \cite{DvorakPo15}. From the definition it follows
that every graph $G$ satisfies
\begin{align} \label{equation:cn<lcn<dcn}
\cn(G:\cP)\leq \lcn(G:\cP)\leq \dcn(G:\cP)
\end{align}
provided that $\cP$ is non-trivial, hereditary, and additive.
The first inequality follows from the fact that a $\cP$-coloring of a graph $G$ with color set $C$ may be considered as a $(\cP,L)$-coloring of $G$ for the constant list assignment $L\equiv C$. To see the second inequality, suppose that $\dcn(G:\cP)=k$ and let $L$ be a list assignment for $G$ with $|L(v)|\geq k$ for all $v\in V(G)$. Define $(X,H)$ to be the cover of $G$ such that $X_v=\{v\}\times L(v)$ for all $v\in V(G)$ and, for two distinct vertices $(u,c)$ and $(v,c')$ of $H$, we have
$$\mu_H((u,c),(v,c'))=
\left\{ \begin{array}{ll}
\mu_G(u,v) & \mbox{\rm if } c=c',\\
0 & \mbox{\rm if } c\not=c'.
\end{array}
\right.$$
We say that $(X,H)$ is the cover \DF{associated} with the list assignment $L$. It is easy to check that $(X,H)$ is indeed a $k$-cover of $G$, and $(X,H)$ has a $\cP$-transversal if and only if $G$ admits an $(\cP,L)$-coloring. This implies, in particular, that $\lcn(G:\cP)\leq k$. Note that the additivity of $\cP$ is only needed for the second inequality.

We call a graph property \DF{reliable} if it is non-trivial, hereditary and additive. In what follows we shall focus mainly on such properties. Suppose that $\cP$ is a reliable graph property and $G$ is an arbitrary graph. Then
\begin{align}
\label{equation:subgraph}
G' \subseteq G \mbox{ implies } \dcn(G':\cP) \leq \dcn(G:\cP).
\end{align}
This follows from the fact that a $k$-cover $(X',H')$ of $G'$ can be extended to a $k$-cover $(X,H)$ of $G$ such that $H'$ is obtained from $H$ by deleting all sets $X_v$ with $v\in V(G)\sm V(G')$. Hence, if $T$ is a $\cP$-transversal of $(X,H)$, then $T'=T \cap V(H')$ is a $\cP$-transversal of $G'$, since $H'[T']$ is an induced subgraph of $H[T]$ and $\cP$ is hereditary. Since $\cP$ is additive, it then follows from \eqref{equation:subgraph} that
\begin{align}
\label{equation:component}
\dcn(G:\cP)= \max \set{\dcn(G':\cP)}{G' \mbox{ is a component of } G}.
\end{align}
Furthermore, we claim that the deletion of any vertex or edge of $G$ decreases the $\cP$-DP-chromatic number of $G$ by at most $\mu(G)$. If $e\in E_G(u,v)$, then $G-v$ is a subgraph of $G-e$. Hence it suffices to show that every vertex $v$ of $G$ satisfies
\begin{align}
\label{equation:deletion}
\dcn(G:\cP)-\mu(G) \leq \dcn(G-v:\cP) \leq \dcn(G:\cP).
\end{align}
The second inequality follows from \eqref{equation:subgraph}. To see the first inequality define $k=\dcn(G-v:\cP)$ and let $(X,H)$ be a $(k+\mu(G))$-cover of $G$. Let $x\in X_v$ and let $(X',H')$ be the cover of $G'$ such that $X_u'=X_u\sm N_H(x)$ for all $u\in V(G')$ and $H'=H-(X_v \cup N_H(x))$. By (C2), $(X',H')$ is a $k$-cover of $G'$ and, therefore, $(X',H')$ has a $\cP$-transversal $T'$. Then $T=T'\cup \{x\}$ is a $\cP$-transversal of $(X,H)$, since $\cP$ is reliable and $H[T]$ is the disjoint union of $H'[T']$ and $K_1$. Consequently, $\dcn(G:\cP)\leq k+\mu(G)=\dcn(G-v:\cP)+\mu(G)$. This proves \eqref{equation:deletion}.

We say that $G$ is \DF{$(\cP,\dcn)$-critical} if every proper induced subgraph $G'$ of $G$ satisfies $\dcn(G':\cP)<\dcn(G:\cP)$. By \eqref{equation:component} it follows that every $(\cP,\dcn)$-critical graph is empty or connected.

\begin{proposition}
Let $\cP$ be a reliable graph property and let $G$ be a graph. Then $G$ has an induced subgraph $G'$ such that $\dcn(G':\cP)=\dcn(G:\cP)$ and $G'$ is $(\cP,\dcn)$-critical.
\label{prop:critsub}
\end{proposition}
\begin{proof}
Among all induced subgraphs $G'$ of $G$ satisfying $\dcn(G':\cP)=\dcn(G:\cP)$ we choose one whose order is minimum. Then $G'$ has the desired properties.
\end{proof}

The above proposition implies that many problems related to the $(\cP,\dcn)$-chromatic number can be reduced to problems about $(\cP,\dcn)$-critical graphs. The study of critical graphs with respect to the ordinary chromatic number was initiated by Dirac in the 1950s (see e.g. \cite{Dirac53} and \cite{Dirac57}) and has attracted a lot of attention until today.

Let $G$ be a graph, and let $(X,H)$ be a cover of $G$. Given a vertex $v\in V(G)$, a partial transversal $T$ of $(X,H)$ such that $\dom(T:G)=V(G-v)$ and $H[T]\in \cP$ is said to be a \DF{$(\cP,v)$-transversal} of $(X,H)$. We call $(X,H)$ a \DF{$\cP$-critical cover} of $G$ if $(X,H)$ has no $\cP$-transversal, but for every vertex $v\in V(G)$ there exists a $(\cP,v)$-transversal. Note that if $G$ is a $(\cP,\dcn)$-critical graph with $\dcn(G:\cP)=k$, then $\dcn(G-v:\cP)\leq k-1$ for all $v\in V(G)$ and, therefore, $G$ has a $\cP$-critical $(k-1)$-cover.

Note that if $\cP$ is a reliable graph property, then any graph in $\CR(\cP)$ is connected. Furthermore, since $K_1\in \cP$ (by Proposition~\ref{prop:smooth}(a)), this implies that
$d(\cP)\geq 1$ and $\cO\subseteq \cP$.

\begin{proposition}
\label{prop:low+vertex}
Let $\cP$ be a reliable graph property with $d(\cP)=r$, let $G$ be graph, and let $(X,H)$ be a $\cP$-critical cover of $G$. Then the following statements hold:
\begin{itemize}
\item[{\rm (a)}] $d_G(v)\geq r|X_v|$ for all $v\in V(G)$.
\item[{\rm (b)}] Let $v$ be a vertex of $G$ such that $d_G(v)=r|X_v|$, and let $T$ be a $(\cP,v)$-transversal of $(X,H)$. Moreover, for $x\in X_v$, let
    $$H_x=H[T \cup \{x\}] \mbox{ and } d_x=d_{H_x}(v).$$
    Then $d_x=r$ for all $x\in X_v$ and $d_G(v)=\sum_{x\in X_v}d_x$.
\end{itemize}
\end{proposition}
\begin{proof}
Let $v$ be an arbitrary vertex of $G$. Since $(X,H)$ is a $\cP$-critical cover of $G$, there is a  $(\cP,v)$-transversal of $G$. Let $T$ be an arbitrary $(\cP,v)$-transversal of $G$. Since $(X,H)$ has no $\cP$-transversal, $H_x=H[T\cup \{x\}]\not\in \cP$ for all $x\in X_v$. Then Proposition~\ref{prop:smooth}(e) implies that $d_x=d_{H_x}(x) \geq d(\cP)=r$ for all $x\in X_v$. Since $|T\cap X_u|=1$ for all $u\in V(G-v)$, we then obtain from (C2) that
$$d_G(v)=|E_G(v)|\geq \sum_{x\in X_v} |E_{H_x}(x)|=\sum_{x\in X_v}d_x\geq r|X_v|$$
Then $d_G(v)=r|X_v|$ implies that $d_x=r$ for all $x\in X_v$. Thus (a) and (b) are proved.
\end{proof}

Let $\cP$ be a reliable graph property with $d(\cP)=r$, let $G$ be a graph, and let $(X,H)$ be a $\cP$-critical cover of $G$. Then define
$$V(G,X,H,\cP)=\set{v\in V(G)}{d_G(v)=r|X_v|}$$
A vertex $v\in V(G)$ is said to be a \DF{low vertex} if $v\in V(G,X,H,\cP)$, and a \DF{high vertex}, otherwise. By the above proposition, every high vertex $v$ of $G$ satisfies $d_G(v)\geq r|X_v|+1$.
Moreover, we call $G[V(G,X,H,\cP)]$ the \DF{low vertex subgraph} of $G$ with respect to $(X,H,\cP)$.

The next result, which is one of our main results in this paper, characterizes the block structure of the low vertex subgraph of cover critical graphs. For covers associated with list assignments of simple graphs, this result was obtained in 1995 by Borowiecki, Drgas-Burchardt and Mih\'ok \cite[Theorem 3]{BorowieckiDrMi95}. First we need some notation. If $G$ is a graph and $t\in \nat$, then $G'=tG$ denotes the graph that results from $G$ by replacing each of its edges with $t$ parallel edges, that is, $V(G')=V(G)$ and $\mu_{G'}(u,v)=t\mu_G(u,v)$ for any two distinct vertices $u, v$ of $G$. A graph $G$ is called a \DF{brick} if $G=tK_n$ with $t,n\in \nat$, or $G=tC_n$ with $t,n\in \nat$ and $n\geq 3$. The proof of the next result is given at the end of Section~\ref{sec:degeneracy}.

\begin{theorem}
\label{theorem:lowvertexA1}
Let $\cP$ be a reliable graph property with $d(\cP)=r$, let $G$ be a graph, and let $(X,H)$ be a $\cP$-critical cover of $G$. If $B$ is a block of the low vertex subgraph $G[V(G,X,H,\cP)]$ of $G$, then $B$ is a brick, or $B=tB'$ with $t\in \nat$ such that either $B'\in \CR(\cP)$ and $B'$ is $r$-regular, or $B'\in \cP$ and $\De(B')\leq r$.
\end{theorem}

In 1963, Gallai~\cite[Satz (E1)]{Gallai63a} characterized the low vertex subgraph of simple graphs being critical with respect to the ordinary chromatic number. He proved that each block of such a low vertex subgraph is a complete graph or an odd cycle, thereby extending Brooks' famous theorem in \cite{Brooks41}. That this also holds for list critical simple graphs was proved by Thomassen \cite{Thomassen97}, an extension to list critical simple hypergraphs was given by Kostochka and Stiebitz \cite{KostStieb03}.  For simple graphs, both results are special cases of Theorem~\ref{theorem:lowvertexA1} by putting $\cP=\cO$ and by choosing covers associated either with constant list assignments or with arbitrary list assignments.

\begin{corollary}
Let $\cP$ be a reliable graph property with $d(\cP)=r$. Then the following statements hold:
\begin{itemize}
\item[{\rm (a)}] If $G$ is a $(\cP,\dcn)$-critical graph with $\dcn(G:\cP)=k+1$ and $k\geq 0$, then $\de(G)\geq rk$. Moreover, if $U=\set{v\in V(G)}{d_G(v)=rk}$, then each block $B$ of $G[U]$ satisfies that $B$ is a brick, or $B=tB'$ with $t\in \nat$ such that either $B'\in \CR(\cP)$ and $B'$ is $r$-regular, or $B'\in \cP$ and $\De(B')\leq r$.
\item[{\rm (b)}] Every graph $G$ satisfies $\dcn(G:\cP)\leq \frac{\De(G)}{r}+1$.
\end{itemize}
\label{corol:gallai}
\end{corollary}
\begin{proof}
To prove (a), note that the assumptions imply that $G$ has a $\cP$-critical $k$-cover, say $(X,H)$ such that $|X_v|=k$ for all $v\in V(G)$. Then $d_G(v)\geq r|X_v|= rk$ for all $v\in V(G)$ (by Proposition~\ref{prop:low+vertex}). Hence $\de(G)\geq rk$ and $U=V(G,X,H,\cP)$ and, therefore, the statements about the blocks in $\cB(G[U])$ are implied by Theorem~\ref{theorem:lowvertexA1}. To prove (b), let $G$ be an arbitrary graph with $\dcn(G:\cP)=k+1$. Then there is a $(\cP,\dcn)$-critical graph $G'$ with $G'\subseteq G$ and $\dcn(G':\cP)=k+1$ (by Proposition~\ref{prop:critsub}). Hence, $\De(G)\geq \De(G')\geq \de(G')\geq rk$ (by (a)), which leads to $\dcn(G:\cP)=k+1\leq \De(G)/r+1$.
\end{proof}

For the ordinary DP-chromatic number (i.e. for $\cP=\cO$), Corollary~\ref{corol:gallai}(a) was proved by Bernshteyn, Kostochka, and Pron \cite{BernsteynKoPro17}; they indeed proved Theorem~\ref{theorem:lowvertexA1} for $\cP=\cO$. Since $\CR(\cO)=\langle K_2 \rangle$ and $d(\cP)=1$, the only type of blocks that can occur in this case are bricks. As noticed by Bernshteyn, Kostochka, and Pron  \cite{BernsteynKoPro17}, for $t,n\in \nat$ with $n\geq 3$, we have $\dcn(tC_n)=2t+1$ even in the case when $\mdo{n}{0}{2}$.

For a reliable graph property $\cP$  and a graph $G$, we have $\dcn(G:\cP)=0$ if and only if $|G|=0$; and $\dcn(G:\cP)=1$ if and only if $G\in \cP$. Furthermore, $G\in \CR(\cP)$ if and only if $G$ is $(\cP,\dcn)$-critical and $\dcn(G:\cP)=2$ (Proposition~\ref{prop:smooth}(b)). Next, we want to establish a Brooks type result for the $\cP$-DP-chromatic number. The case $\cP=\cO$ of the following result was obtained by Bernshteyn, Kostochka, and Pron \cite{BernsteynKoPro17}.

\begin{theorem}
\label{theorem:Brooks}
Let $\cP$ be a reliable graph property with $d(\cP)=r$, and let $G$ be a connected simple graph. Then
\begin{align}
\label{equation:brooks}
\dcn(G:\cP)\leq \left\lceil \frac{\De(G)}{r} \right\rceil,
\end{align}
unless $G=K_{kr+1}$ for some integer $k\geq 0$, or $G$ is $r$-regular and $G\in \CR(\cP)$, or $\cP=\cO$ and $G$ is a cycle.
\end{theorem}
\begin{proof}
Let $G$ be a connected graph. If $\De(G)$ is not divisible by $r$, then \eqref{equation:brooks} is an immediate consequence of Corollary~\ref{corol:gallai}(b) and we are done. So assume that $\De(G)=kr$ for some integer $k\geq 0$. Then $\dcn(G:\cP)\leq k+1$ (by Corollary~\ref{corol:gallai}(b)). In the case that $\dcn(G:\cP)\leq k$ we are done, too. Hence it remains to consider the case that $\dcn(G:\cP)= k+1$. Then $G$ has an induced subgraph $G'$ such that
$G'$ is $(\cP,\dcn)$-critical and $\dcn(G':\cP)= k+1$ (by Proposition~\ref{prop:critsub}). Then $\de(G')\geq rk$ (by Corollary~\ref{corol:gallai}(b)) and, since $G$ is connected and $\De(G')\leq \De(G)=rk$, we obtain that $G=G'$ and so $G$ is regular of degree $rk$. This implies that the set of low vertices $U=\set{v\in V(G)}{d_G(v)=rk}$ satisfies $U=V(G)=V(G')$ and so $G=G'[U]$.
Since $G$ is a simple graph, it then follows from Theorem~\ref{theorem:lowvertexA1} that $G$ is a complete graph, or $G$ is a cycle, or $G$ is $r$-regular and $G\in \CR(\cP)$, or $G\in \cP$ and $\De(G)\leq r$. Since $G$ is regular of degree $kr$, we conclude that $G$ itself is a block, unless $k=1$. If $k=1$, then $G$ is $r$-regular and $\dcn(G:\cP)=2$, which implies that $G\not\in \cP$. Since $G$ is $(\cP,\dcn)$-critical, $\dcn(G-v:\cP)\leq 1$ for every $v\in V(G)$, and so $G-v\in \cP$ for every $v\in V(G)$. Consequently, $G\in \CR(\cP)$ and we are done. Now assume that $k\not=1$. Then $G$ is a block and so $G$ is a complete graph or a cycle.

If $G$ is a $K_n$, then $n-1=kr$ and we are done.
It remains to consider the case that $G$ is a cycle. Since $G$ is $kr$-regular and $k\not=1$, this implies that $k=2$, $r=1$ and $\dcn(G:\cP)=3$. Since $\cP$ is reliable, $\cO\subseteq \cP$. If $K_2\in \CR(\cP)$ then $\cP=\cO$ (by Proposition~\ref{prop:smooth}(b)(c)) and we are done, too. Otherwise
$K_2\in \cP$, and it is not difficult to show that $\dcn(G:\cP)\leq 2$, a contradiction. Let $(X,H)$ be an arbitrary cover of $G$ such that $|X_v|=2$ for all $v\in V(G)$. It suffices to show that there exists a transversal $T$ of $(X,H)$ such that $H[T]\in \cP$. If $(X,H)$ has an independent transversal, this is obviously true. Otherwise, it follows from \cite[Theorem 2]{Schweser18b} that if $e$ is an edge of $H$, then there exists a transversal $T$ of $(X,H)$ such that $e$ is the only edge of $H[T]$ and so $H[T]\in \cP$.
This completes the proof.
\end{proof}

Note that the above theorem for $\cP=\cO$ implies Brooks' famous theorem \cite{Brooks41} from 1941 saying that any connected simple graph $G$ satisfies $\cn(G)\leq \De(G)$ unless $G$ is a complete graph or an odd cycle (use \eqref{equation:cn<lcn<dcn} and the trivial fact that any even cycle has $\cn=2$).

The next result is an extension of a well known result  about degree choosable graphs due to Erd\H{o}s, Rubin, and Taylor \cite{ErdosRubTay79}, it was independently proved by Oleg Borodin in his thesis (Problems of coloring and of covering the vertex set of a graph by induced subgraphs, Novosibirsk 1979). For $\cP=\cO$, the next result was obtained by  Bernshteyn, Kostochka, and Pron \cite{BernsteynKoPro17}. Note that $K_2$ is the only graph in $\CR(\cO)$ and $K_1$ is the only block belonging to $\cO$.

\begin{theorem}
\label{theorem:ERT}
Let $\cP$ be a reliable graph property with $d(\cP)=r$, let $G$ be a connected graph, and let $(X,H)$ be a cover of $G$ such that $r|X_v|\geq d_G(v)$ for all $v\in V(G)$. If $G$ is not $(\cP,(X,H))$-colorable, then each block $B$ of $G$ is a brick, or $B=tB'$ with $t\in \nat$ such that either $B'\in \CR(\cP)$ and $B'$ is $r$-regular, or $B'\in \cP$ and $\De(B')\leq r$.
\end{theorem}
\begin{proof}
By assumption, $(X,H)$ has no $\cP$-transversal. Then there is a vertex set $U\subseteq V(G)$ such that the cover $(X',H')$ of $G'=G-U$ with $H'=H-\bigcup_{u\in U}X_u$ and $X'=X|_{V(G)\sm U}$ is  $\cP$-critical. By Proposition~\ref{prop:low+vertex}, we have $d_{G'}(u)\geq r|X'_u|=r|X_u|\geq d_{G}(u)$ for all $u\in V(G')$. Since $G$ is connected, this implies that $G=G'$ and so $(X,H)$ is a $\cP$-critical cover of $G$. Moreover it follows that $r|X_v|=d_G(v)$, from which we obtain that $V(G,X,H,\cP)=V(G)$, that is, $G$ is its own low vertex subgraph. Then Theorem~\ref{theorem:lowvertexA1} implies the required properties for the blocks of $G$.
\end{proof}

\section{DP-Coloring and variable degeneracy}
\label{sec:degeneracy}

For proving Theorem~\ref{theorem:lowvertexA1}, we shall establish a result (Theorem~\ref{main_theorem}) that combines DP-coloring with variable degeneracy. Let $H$ be a graph and let $f$ be a \DF{vertex function} of $H$, i.e. $f:V(H)\to \nato$. Then $\su(f)=\set{x\in V(H)}{f(x)>0}$ is the \DF{support} of $f$ in $H$, and $\suc(f)=\set{x\in V(H)}{f(x)=0}$ is the \DF{complementary support}
of $f$ in $H$. For a set $X\subseteq V(H)$, define
$$f(X)=\sum_{x\in X}f(x).$$
A subgraph $\tilde{H}$ of $H$ is called \DF{strictly $f$-degenerate} if each non-empty subgraph $H'$ of $\tilde{H}$ contains a vertex $x$ such that $d_{H'}(x)<f(x)$. Note that if a subgraph $\tilde{H}$ of $H$ is a strictly $f$-degenerate, then $V(\tilde{H})\subseteq \su(f)$. The concept of variable degeneracy seems to have been first studied by Borodin, Kostochka, and Toft \cite{BorodinKT00}. DP-colorings with variable degeneracy for simple graphs were introduced by Sittitrai and Nakprasit~\cite{SittitraiNa18} although they use a slightly different approach.

In this section we study the following coloring problem. A \DF{configuration} is a tuple $\cK=
(G,X,H,f)$ such that $G$ is a graph, $(X,H)$ is a cover of $G$, and $f$ is a vertex function of $H$. Given a configuration $\cK=(G,X,H,f)$, we want to decide whether $(X,H)$ has a transversal $T$ such that $H[T]$ is strictly $f$-degenerate. In general, this decision problem is {\NPC}. However, if we add a certain degree condition it might become a polynomial problem.

Let $\cK=(G,X,H,f)$ be a configuration. We call $\cK$  \DF{degree-feasible} if for each vertex $v$ of $G$ we have
$$f(X_v)=\sum_{x \in X_v} f(x) \geq d_G(v).$$
Furthermore, we say that $\cK$ is \DF{colorable} if $(X,H)$ has a transversal $T$ such that $H[T]$ is strictly $f$-degenerate, otherwise $\cK$ is said to be \DF{uncolorable}. If we want to decide whether $\cK$ is colorable, or not, we always assume that $|X_v|=r$ for all $v\in V(G)$ with $r\geq 1$, for otherwise we may add virtual vertices $x$ and put $f(x)=0$. In what follows, we shall use this assumption in order to simplify our description. Our aim is to characterize degree feasible uncolorable configurations.

\begin{figure}[htbp]
\centering
\includegraphics[height=5cm]{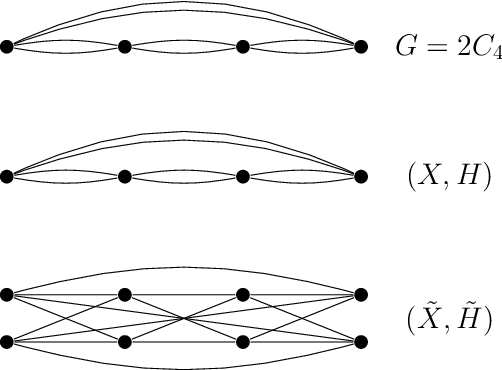}
\hspace{1cm}
\includegraphics[height=5cm]{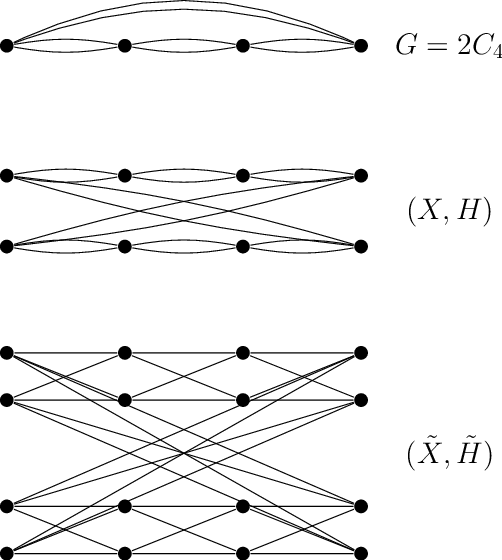}
%
%
\caption{A $G$-saturated cover $(X,H)$ of $G$ and its $2$-inflation $(\tilde{X},\tilde{H})$.}
\label{figure1}       
\end{figure}

First, we need some more notation. Let $G$ be a non-empty graph, and let $(X,H)$ be a cover of $G$. For every edge $uv$ of $G^o$, $H(X_u,X_v)$ is a bipartite graph with parts $X_u$ and $X_v$ and $\De(H(X_u,X_v))\leq \mu_G(u,v)$ (by (C2)). We say that $(X,H)$ is \DF{$G$-saturated} if for every
edge $uv$ of $G^o$, the bipartite graph $H(X_u,X_v)$ is regular of degree $\mu_G(u,v)$, see Fig.~\ref{figure1}. If the cover $(X,H)$ is $G$-saturated and $G$ is connected, then there is an integer $r\in \nat$ such that $|X_v|=r$ for all $v\in V(G)$; in this case we say that $(X,H)$ is an \DF{$r$-uniform} cover of $G$. Let $U\subseteq V(H)$ be an arbitrary set. For $v\in V(G)$, we define $X_v(U)=X_v\cap U$. Let $G'=G[\dom(U:G)]$, let $H'=H[U]$, and let $X':V(G') \to 2^U$ be the map with $X'_v=X_v(U)$ for all $v\in V(G')$. Then $(X',H')$ is a cover of $G'$ and we write $(X',H')=(X,H)/U$, we call this cover a \DF{subcover} of $(X,H)$ \DF{restricted} to $U$. If $\dom(U:G)=V(G)$, then $(X',H')$ is a cover of $G$, in this case we say that $(X',H')$ is a \DF{full subcover} of $(X,H)$ \DF{restricted} to $U$. If $(X',H')=(X,H)/U$ is a full subcover of $(X,H)$ restricted to $U$ and $(X',H')$ is $G$-saturated, then $E_H(U,V(H)\sm U)=\ems$ (by (C1) and (C2)).

Let $G$ be a non-empty graph, let $(X,H)$ be a cover of $G$, and let $s\in \nat$ be an integer such that $\mdo{\mu_H(x,y)}{0}{s}$ for every pair $(x,y)$ of distinct vertices of $H$. We now construct a new cover $(\tilde{X}, \tilde{H})$ of $G$ as follows:
\begin{itemize}
  \item For every vertex $x$ of $H$ let $U_x$ be a set of $s$ new vertices, and let $X_v'=\bigcup_{x\in X_v}U_x$ for all $v\in V(G)$; note that $|X'_v|=s|X_v|$.
  \item For every pair $x, y$ of distinct vertices of $H$, let $\tilde{H}_{x,y}=\tilde{H}[U_x \cup U_y]$ be a copy of the bipartite graph $mK_{s,s}$ with parts $U_x$ and $U_y$ where $m=\mu_H(x,y)/s$. Let $\tilde{H}$ denote the union of all these graphs $\tilde{H}_{x,y}$.
\end{itemize}
It is easy to show that $(\tilde{X}, \tilde{H})$ is a cover of $G$; we call this cover  \DF{$s$-inflation} of $(X,H)$ (see Fig.~\ref{figure1}). By an inflation of $(X,H)$ we mean an $s$-inflation of $(X,H)$ with $s\in \nat$, which only exists if $H=sH'$. Clearly, if $(X,H)$ is $G$-saturated, so is each of its inflations.

Given a graph $G$, we define two special types of covers of $G$. A cover $(X,H)$ of $G$ is called a \DF{$G$-cover} if $X_v=\{x_v\}$ for all $v\in V(G)$ and $\mu_H(x_u,x_v)=\mu_G(u,v)$ for every edge $uv$ of $G^o$. Note that $H$ is a copy of $G$ with $x_v \mapsto v$ as isomorphism, and so $d_H(x_v)=d_G(v)$ for all $v\in V(G)$. Furthermore, $(X,H)$ is $G$-saturated. A cover $(X,H)$ of $G$ is called a \DF{double $G$-cover} if $G=tC_n$ with $n\geq 3$ and $t\in \nat$ for a cycle $C_n=(v_1,v_2, \ldots,v_n, v_1)$, $X_{v_i}$ ($i\in [1,n]$) is a set of two vertices, say $x_i, x_{n+i}$, and $\mu_H(x_i,x_j)=t$ if $\mdo{i-j}{1}{2n}$ else $\mu_H(x_i,x_j)=0$. Clearly, $(X,H)$ is $G$-saturated. Note that the cover $(X,H)$ shown on the right side of Fig.~\ref{figure1} is a double $G$-cover for $G=2C_4$.

\smallskip

Next we recursively define the family of constructible configurations. A uniform configuration $(G,X,H,f)$ is called \DF{constructible} if one of the following five conditions hold:

\begin{description}
\item[{\rm (K1)}] $(G,X,H,f)$ is an \upshape{M}-\DF{configuration}, that is, $G$ is a block and there exists a set $U\subseteq V(H)$ such that $(X,H)/U$ is an inflation of a $G$-cover, and, for $v\in V(G)$ and $x\in X_v$,we have $f(x)=d_G(v)/|X_v(U)|$ if $x\in U$ else $f(x)=0$. The set $U$ is called \DF{layer} of $(X,H)$.
    \\\\
    {\bf Remark 1} Note that the cover $(X,H)/U$ is a full subcover of $(X,H)$ that is $G$-saturated, and, for $v\in V(G)$, we have $f(X_v)=f(X_v(U))=d_G(v)$. Hence any M-configuration is  degree feasible.

\item[{\rm (K2)}] $(G,X,H,f)$ is a \upshape{K}-\DF{configuration}, that is, $G=tK_n$ with $t,n\in \nat$ and there are integers $n_1,n_2,\ldots,n_p\in \nat$ with $p \geq 1$ such that $n_1 + n_2 + \ldots + n_p=n-1$. Moreover, there are $p$ disjoint subsets $U_1, U_2, \ldots, U_p$ of $V(H)$ such that $(X,H)/U_i$ is an inflation of a $G$-cover for $i\in [1,p]$, and, for $v\in V(H)$ and $x\in X_v$, we have $f(x)=(tn_i)/|U_i(X_v)|$ if $x\in U_i$ for $i\in [1,p]$ else $f(x)=0$. The set $U_i$ is said to be a \DF{layer} of $(G,X,H,f)$ of \DF{type} $n_i$ ($i\in [1,p]$).
    \\\\
    {\bf Remark 2} Note that the cover $(X,H)/U_i$ is $G$-saturated for $i\in [1,p]$, and, for $v\in V(G)$, we have $f(X_v(U_i))=tn_i$ and so $f(X_v)=t(n-1)=d_G(v)$. Hence any K-configuration is degree feasible. Furthermore, a K-configuration with $p=1$ is also an M-configuration.

\item[{\rm (K3)}] $(G,X,H,f)$ is an \DF{odd} \upshape{C}-\DF{configuration}, that is, $G=tC_n$ with $t,n\in \nat$ and $n\geq 3$ odd. Moreover, there are two disjoint subsets $U_1, U_2$ of $V(H)$ such that $(X,H)/U_i$ is an inflation of a $G$-cover for $i\in\{1,2\}$, and,
    for $v\in V(G)$ and $x\in X_v$, we have $f(x)=t/|U_i(X_v)|$ if $x\in U_i$ for $i\in \{1,2\}$ else $f(x)=0$.
    \\\\
    {\bf Remark 3} Note that the cover $(X,H)/U_i$ is $G$-saturated for $i\in \{1,2\}$, and, for $v\in V(G)$, we have $f(X_v(U_i))=t$ and so $f(X_v)=2t=d_G(v)$. Hence any odd C-configuration is  degree feasible. Furthermore, if $n=3$ then $(G,X,H,f)$ is also a K-configuration.

\item[{\rm (K4)}] $(G,X,H,f)$ is an \DF{even} \upshape{C}-\DF{configuration}, that is, $G=tC_n$ with $t,n\in \nat$ and $n\geq 4$ even. Moreover, there is a subset $U$ of $H$ such that $(X,H)/U$ is an inflation of a double $G$-cover, and, for $v\in V(G)$ and $x\in X_v$, we have $f(x)=2t/|U(X_v)|$ if $x\in U$ else $f(x)=0$.
    \\\\
    {\bf Remark 4} Note that the cover $(X,H)/U$ is $G$-saturated, and, for $v\in V(G)$, we have $f(X_v)=f(X_v(U))=2t=d_G(v)$. Hence any even C-configuration is  degree feasible.

\item[{\rm (K5)}] There are two disjoint constructible configurations, say  $(G^1,X^1,H^1,f^1)$ and $(G^2,X^2,H^2,f^2)$, such that $G$ is obtained from the disjoint graphs $G^1$ and $G^2$ by identifying a vertex $v^1 \in V(G^1)$ and a vertex $v^2 \in V(G^2)$ to a new vertex $v^*$, $H$ is obtained from the disjoint graphs $H^1$ and $H^2$ by choosing a bijection $\pi$ from $X_{v^1}$ to $X_{v^2}$ and identifying each vertex $x \in X_{v_1}$ with $\pi(x)$ to a vertex $x^*$, and $f$ is defined as
$$f(y)=\begin{cases}
f^1(y) & \text{if } y \in V(H^1) \setminus X_{v^1},\\
f^2(y) & \text{if } y \in V(H^2) \setminus X_{v^2},\\
f^1(x) + f^2(\pi(x)) &\text{if } y \text{ is obtained from the identification of } x \in X_{v^1} \\ &\text{ with } \pi(x) \in X_{v^2}.
\end{cases}$$
In this case, we say that $(G,X,H,f)$ is obtained from $(G^1,X^1,H^1,f^1)$ and $(G^2,X^2,H^2,f^2)$ by \DF{merging} $v^1$ and $v^2$ to $v^*$.
\end{description}

By a \upshape{C}-\DF{configuration} we mean either an odd or an even \upshape{C}-configuration. The next result characterizes uncolorable degree feasible configurations whose underlying graph is connected.

\begin{theorem} \label{main_theorem}
Let $G$ be a connected  graph, let $(X,H)$ be a cover of $G$, and let $f:V(H) \to \mathbb{N}_0$ be a function. Then, $(G,X,H,f)$ is an uncolorable degree-feasible configuration if and only if $(G,X,H,f)$ is constructible.
\end{theorem}

For the class of simple graphs, Theorem~\ref{main_theorem}, formulated in a slightly different terminology, was obtained by F. Lu, Q. Wang and T. Wang, see \cite{Lu2Wang}. For covers $(X,H)$ associated with constant list assignments of $G$, Theorem~\ref{main_theorem} is a reformulation of a result that was obtained in 2000 by Borodin, Kostochka, and Toft
\cite[Theorem 8]{BorodinKT00} for simple graphs and extended in 2021 by Schweser and Stiebitz
\cite[Theorem 2]{SchweserS2021} to graphs and hypergraphs. The proof of Theorem~\ref{main_theorem} resembles the proofs given in \cite{BorodinKT00} and \cite{SchweserS2021}; the proof is done via a sequence of five propositions and one theorem.

Recall that for a graph $G$, we denote by $\cB(G)$ the set of blocks of $G$. Furthermore, for $v\in V(G)$, let $\cB_v(G)=\set{B\in \cB(G)}{v\in V(B)}$. Note that two blocks of $G$ have at most one vertex in common; and a vertex $v$ belongs to $S(G)$ if and only if $v$ belongs to more than one block of $G$. The blocks of $G$ form a tree-like structure, and a block $B$ of $G$ is called an \DF{end-block} of $G$ if $|S(G)\cap V(B)|=1$. Any connected graph $G$ is either block or has at least two end-blocks.

Next, we want to describe the block structure of a constructible configuration. First, we need some notation. Let $G$ be a connected graph, let $(X,H)$ be a cover of $G$, and let $B$ be a non-empty induced subgraph of $G$. Then we denote by $X^B$ the restriction of the map $X$ to $V(B)$, and by $H^B$ we denote the subgraph of $H$ induced by the vertex set $U=\bigcup_{v\in V(B)}X_v$. Clearly, $(X^B,H^B)$ is a cover of $B$, and $(X^B,H^B)=(X,H)/U$. In particular, $X^B_v=X_v$ for every vertex $v$ of $B$, and $H^B(X_u,X_v)=H(X_u,X_v)$ for every edge $uv$ of $B^o$. The proof of the following proposition can be easily done by induction on the number of blocks of the graph $G$. The after next proposition is an immediate consequence; note that an edge of $G$ belongs to exactly one block of $G$.

\begin{proposition} \label{prop_construtible-blocks}
Let $\cK=(G,X,H,f)$ be a constructible configuration. Then for every $B\in \cB(G)$, there is a unique function $f^B$, such that $\cK^B=(B,X^B,H^B,f^B)$ is an \upshape{M}-, \upshape{K}- or \upshape{C}-configuration, and, for $v \in V(G)$ and $x \in X_v$, we have
$$f(x)=\sum_{B \in \cB_v(G)} f^B(x)$$
In what follows, we call $f^B$ the \DF{$B$-part} of the function $f$.
\end{proposition}

\begin{proposition} \label{prop_construtible-blocksB}
Let $\cK=(G,X,H,f)$ be a constructible configuration, let $B\in \cB(G)$, let $f^B$ be the $B$-part of $f$, and let $\cK^B=(B,X^B,H^B,f^B)$. Suppose that there is a pair $u, v$ of distinct vertices of $B$ such that $H(X_u,X_v)$ has only one component. Then $\cK^B$ is an M-configuration, and if $u, v$ are distinct vertices of $B$ with $\mu_H(u,v)=m>0$, then $H(X_u,X_v)$ is a $tK_{s,s}$ with $m=ts$ for $s,t\in \nat$.
\end{proposition}

The next proposition proves the ``if''-direction of Theorem~\ref{main_theorem}.

\begin{proposition} \label{prop_constructible}
Let $(G,X,H,f)$ be a constructible configuration. Then the following statements hold:
\begin{itemize}
\item[\upshape (a)] $f(X_v) = d_G(v)$ for all $v \in V(G)$.
\item[\upshape (b)] $(G,X,H,f)$ is uncolorable.
\end{itemize}
\end{proposition}

\begin{proof}
Statements (a) holds if $G$ is a block (see Remarks 1 - 4); for arbitrary connected graphs $G$ it then easily follows from Proposition~\ref{prop_construtible-blocks}. The proof of (b) is by reductio ad absurdum. Then we may choose a configuration $\cK=(G,X,H,f)$ such that
\begin{itemize}
\item[(1)] $\cK$ is constructible,
\item[(2)] $\cK$ is colorable, i.e., there is a transversal $T$ of $(X,H)$ such that $H[T]$ is strictly $f$-degenerate, and
\item[(3)] $|G|$ is minimum subject to (1) and (2).
\end{itemize}
Note that if $f(x)=0$ for some vertex $x \in V(H)$, then $H[\{x\}]$ is not strictly $f$-degenerate and, hence, $x$ cannot be contained in any strictly $f$-degenerate subgraph of $H$.

First, assume that $(G,X,H,f)$ is an \upshape{M}-configuration. Then $G$ is a block and there exists a set $U\subseteq V(H)$ such that $(X,H)/U$ is an $s$-inflation of a $G$-cover of $G$. Let $v\in V(G)$ and $x\in X_v$. Then $s=|X_v(U)|$, and $f(x)=d_G(v)/s$ if $x\in U$ else $f(x)=0$. Consequently, $T\subseteq U$ and $d_{H[T]}(x)=d_G(v)/s=f(x)$ if $x\in T$, a contradiction.

Next assume that $(G,X,H,f)$ is a \upshape{K}-configuration. Then $G=tK_n$ with $t,n\in \nat$ and there are integers $n_1,n_2,\ldots,n_p \in \nat$ with $p \geq 1$ such that $n_1 + n_2 + \ldots + n_p=n-1$. Moreover, there are $p$ disjoint subsets $U_1, U_2, \ldots, U_p$ of $V(H)$ such that $(X,H)/U_i$ is an inflation of a $G$-cover for $i\in [1,p]$, and for $v\in V(H)$ and $x\in X_v$ we have $f(x)=(tn_i)/|U_i(X_v)|$ if $x\in U_i$ for $i\in [1,p]$ else $f(x)=0$. Then $T$ is a subset of $U_1 \cup U_2 \cup \cdots \cup U_p$. Since $|T|=n$, there is an $i\in [1,p]$, such that $m=|T\cap U_i|\geq n_i+1$. If $(X,H)/U_i$ is an $s$-inflation of $G$, then $s=|U_i(X_v)|$ for all $v\in V(G)$ and $H'=H[T\cap U_i]$ is a $(t/s)K_m$ implying that $d_H'(x)= (t/s)(m-1)\geq (t/s)n_i=f(x)$ for every $x\in V(H')$, a contradiction. If $(G,X,H,f)$ is a \upshape{C}-configuration, we may argue similarly to get a contradiction.

To complete the proof, it remains to consider the case that $\cK=(G,X,H,f)$ is obtained from two constructible configurations $\cK^1=(G^1,X^1,H^1,f^1)$ and $\cK^2=(G^2,X^2,H^2,f^2)$ by merging $v^1 \in V(G^1)$ and $v^2 \in V(G^2)$ to a new vertex $v^*$. To simplify the proof, we assume that $v^1=v^2=v^*$ and $X_{v^*}^1=X_{v^*}^2$ (i.e., $\pi=id$). Since $|G|$ was chosen minimal with respect to (1) and (2), we conclude that $\cK^i$ is uncolorable for $i \in \{1,2\}$. By (2), $(X,H)$ has a transversal $T$ such that $H[T]$ is strictly $f$-degenerate. Let $T^i=T\cap V(H^i)$ ($i\in \{1,2\}$) and let $x_{v^*}$ be the unique vertex from $X_{v^*}\cap T$. Then $T^1\cap T^2=\{x_{v^*}\}$, and $f(x_{v^*})=f^1(x_{v^*}) + f^2(x_{v^*})$ and $f(x)=f^i(x)$ for all $x \in T^i \sm \{x_{v^*}\}$ ($i \in \{1,2\}$). Since $\cK^i$ is uncolorable, the subgraph $H[T^i]$ is not strictly $f^i$-degenerate implying that there is a subgraph $\tilde{H^i}$ of $H[T^i]$ such that $d_{\tilde{H^i}}(x) \geq f^i(x)$ for all $x \in V(\tilde{H^i})$ ($i\in \{1,2\}$). If $x_{v^*}$ does not belong to $\tilde{H^i}$, then $\tilde{H^i}$ is a subgraph of $H[T]-x_{v^*}$ and so
$\tilde{H^i}$ is strictly $f^i$-degenerate as $H[T]$ is strictly $f$-degenerate, which is impossible. Hence $x_{v^*}$ belongs to $\tilde{H^i}$ (for $i\in \{1,2\})$. Then $\tilde{H}=\tilde{H^1} \cup \tilde{H^2}$ is a subgraph of $H[T]$. Let $x$ be an arbitrary vertex of $\tilde{H}$. If $x\not=x_{v^*}$, then $x$ belongs to $\tilde{H^i}-x_{v^*}$ for some $i\in \{1,2\}$, and so $d_{\tilde{H}}(x)=d_{\tilde{H^i}}(x) \geq f^i(x) = f(x)$. Furthermore, we have
$$d_{\tilde{H}}(x_{v^*})=d_{\tilde{H^1}}(x_{v^*}) + d_{\tilde{H^2}}(x_{v^*}) \geq f^1(x_{v^*}) + f^2(x_{v^*}) = f(x_{v^*}).$$ Hence, $\tilde{H} \subseteq H[T]$ is not strictly $f$-degenerate and so $H[T]$ is not strictly $f$-degenerate, as well, a contradiction.
\end{proof}

As a consequence of the above proposition, it only remains to show that each uncolorable degree-feasible configuration is constructible. To this end, we apply the following reduction method.

\begin{proposition}[Reduction]\label{prop_reduction}
Let $\cK=(G,X,H,f)$ be a configuration, let $v \in V(G) \setminus S(G)$ and $x_v \in X_v$ such that $f(x_v)>0$. Moreover, let $\cK'=(G',X',H',f')$ be the tuple with $G'=G-v$, $(X',H')=(X^{G'},H^{G'})$, and
$$f'(x)=\max\{0,f(x)-\mu_H(x,x_v)\}$$
for all $x\in V(H')$.
Then, $\cK'$ is a configuration and the following statements hold:
\begin{itemize}
\item[\upshape (a)] If $\cK$ is degree-feasible, then so is $\cK'$.
\item[\upshape (b)] If $\cK$ is uncolorable, then so is $\cK'$.
\end{itemize}
In the following, we write $\cK'=\cK/(v,x_v)$.
\end{proposition}
\begin{proof}
That $\cK'$ is a configuration is evident, note that $H'=H-X_v$. For the proof of (a) assume that $\cK$ is degree-feasible. Let $u$ be an arbitrary vertex of $G'$. Then we have
$$d_G(u)-d_{G'}(u)=\mu_G(u,v)\geq \sum_{x\in X_u}\mu_H(x,x_v),$$
where the last inequality follows from the fact that both sets $X_u$ and $X_v$ are independent and $E_H(X_u,X_v)$ is a union of $\mu_G(u,v)$ matchings (by (C2)). Then we conclude that
$$\sum_{x\in X_u}f'(x)\geq \sum_{x\in X_u}(f(x)-\mu_H(x,x_v))\geq d_G(u)-\sum_{x\in X_u}\mu_H(x,x_v)\geq d_{G'}(u).$$
Hence $\cK'$ is degree-feasible and (a) is proved. For the proof of (b) assume that $\cK'$ is colorable. Then there is a transversal $T'$ of $(X',H')$ such that $H'[T']$ is strictly $f'$-degenerate. Clearly, $T=T' \cup \{x_v\}$ is a transversal of $(X,H)$. Let $\tilde{H}$ be a non-empty subgraph of $H[T]$. We claim that there is a vertex $x$ in $\tilde{H}$ such that $d_{\tilde{H}}(x)<f(x)$. If $x_v$ is the only vertex of $\tilde{H}$, then $f(x_v)>0=d_{\tilde{H}}(x_v)$ and we are done. Otherwise, $\tilde{H}$ contains a vertex $x\not=x_v$ such that $d_{H[T']}(x)<f'(x)$ since $H'[T']$ is strictly $f'$-degenerate. Then $f'(x)=f(x)-\mu_H(x,x_v)$ and so
$$d_{\tilde{H}}(x)\leq d_{H[T]}(x)+\mu_H(x,x_v)<f'(x)+\mu_H(x,x_v)= f(x)$$
This proves that $H[T]$ is strictly $f$-degenerate. Consequently, $\cK$ is colorable and (b) is proved.
\end{proof}

By using the reduction method, we obtain the following useful properties of uncolorable degree-feasible configurations.

\begin{proposition}\label{prop_uncolorable-facts}
Let $G$ be a connected graph and let $(G,X,H,f)$ be an uncolorable degree-feasible configuration. Then, the following statements hold:
\begin{itemize}
\item[\upshape(a)] $f(X_u)=d_G(u)$ for all $u \in V(G)$.
\item[\upshape(b)] If $v\in V(G)\sm S(G)$ and $x_v \in X_v$ with $f(x_v) > 0$, then, for every $u\in V(G-v)$, we have $f(y)\geq \mu_H(y,x_w)$ for all $y\in X_u$ and $\mu_G(v,u)= \sum_{x\in X_v}\mu_H(y,x_v)$.
\item[\upshape(c)] If $|G| \geq 2$ and if $u$ is an arbitrary vertex of $G$, then there is a partial transversal $T$ of $(X,H)$ such that $\dom(T:H)=V(G-u)$ and $H[T]$ is strictly $f$-degenerate, and, for every such transversal $T$ and every vertex $x\in X_u$, we have $f(x)=d_{H[T\cup \{x\}]}(x)$.
\end{itemize}

\end{proposition}
\begin{proof}
The proof of (a) is by induction on the order of $G$. If $|G|=1$ then $\suc(f)=V(H)$ and the statement is obvious. Suppose $|G| \geq 2$ and let $u \in V(G)$ be an arbitrary vertex. Since $G$ is connected, there is a non-separating vertex $v \neq u$ in $G$ and, since $f(X_v) \geq d_G(v) \geq 1$, there is at least one vertex $x_v \in X_v$ with $f(x_v)>0$. By Proposition~\ref{prop_reduction}, $(G',X',H',f')=(G,X,H,f)/(v,x_v)$ is an uncolorable degree-feasible configuration, where $G'=G-v$ and $X'_w=X_w$ for all vertices $w$ of $G'$.
By applying the induction hypothesis, we obtain
$$f'(X_u)=f'(X'_u)=d_G'(u).$$
Since $G'=G-v$, we have
$$d_G(u)-d_{G'}(u)=\mu_G(v,v)\geq \sum_{y\in X_u}\mu_H(y,x_v),$$
where the last inequality follows from (C2).
Furthermore, $f'(y)\geq f(y)-\mu_H(y,x_v)$ for all $y\in X_u$. Consequently, we obtain that
$$d_{G'}(u)=\sum_{y\in X_u}f'(y)\geq \sum_{y\in X_u}(f(y)-\mu_H(y,x_v))\geq d_G(u)-\sum_{y\in X_u}\mu_H(y,x_v)\geq d_{G'}(u),$$
which implies that $\sum_{y\in X_u}f(y)=d_G(u)$, $f'(y)=f(y)-\mu_H(y,x_v)\geq 0$ for all $y\in X_u$ and $\mu_G(v,u)= \sum_{y\in X_u}\mu_H(y,x_v)$.
This proves (a) and (b).

For the proof of (c), assume that $|G| \geq 2$ and $u \in V(G)$. Let $G'=G-u$, let $X'$ be the restriction of $X$ to $V(G')$, and let $H'=H-X_u$. Then $(G',X',H',f)$ is a degree-feasible configuration. Furthermore, each component of $G'$ contains a vertex $v\in N_G(u)$ and we then obtain from (a) that
$$f(X_v)= d_G(v) > d_{G'}(v).$$
Then it follows from (a) applied to each component of $G'$ that $(X',H')$ has a transversal $T$ such that $H'[T]$ is strictly $f$-degenerate. Then $T$ is a partial transversal  of $(X,H)$ such that $\dom(T:H)=V(G')$ and $H[T]=H'[T]$ is strictly $f$-degenerate. Now let $T$ be such a partial transversal. Since $(G,X,H,f)$ is uncolorable, for each $x \in X_u$, $H[T \cup \{x\}]$ contains a non-empty subgraph $H_x$ such that $f(y) \leq d_{H_x}(y)$ for all $y \in V(H_x)$. Clearly, $H_x$ contains $x$ and so $f(x) \leq d_{H_x}(x) \leq d_{H[T\cup\{x\}]}(x)$. For a vertex $y\in T$, let $v=v_y$ be the unique vertex of $G'$ such that $y\in X_v$. Using (a) and (C2), we then obtain that
\begin{eqnarray*}
d_G(u)&=&\sum_{x\in X_u}f(x)\leq \sum_{x\in X_u}d_{H[T\cup\{x\}]}(x)=\sum_{x\in X_u}\sum_{y\in T}\mu_H(x,y)\\
&=& \sum_{y\in T}\sum_{x\in X_u}\mu_H(x,y)\leq \sum_{y\in T}\mu_G(u,v_y)=\sum_{v\in V(G')}\mu_G(u,v)=d_G(u).
\end{eqnarray*}
This obviously implies that $f(x)=d_{H[T\cup \{x\}]}(x)$ for all $x\in X_u$. This proves (c).
\end{proof}

Let $\cK=(G,X,H,f)$ be a configuration. Then we call $G$ the \DF{fundamental graph} of $\cK$, $(X,H)$ the \DF{cover} of $\cK$, and $f$ the \DF{function} of $\cK$. The next theorem proves the ``only if''-direction of Theorem~\ref{main_theorem}.

\begin{theorem} \label{theorem:main2}
If $\cK$ is an uncolorable degree-feasible configuration whose fundamental graph is connected, then $\cK$ is colorable.
\end{theorem}

\begin{proof}
The proof is by reductio ad absurdum. Let $\cK=(G,X,H,f)$ be a minimal counter-example, that is, $G$ is a connected graph such that
\begin{itemize}
\item[(A)] $\cK$ is an uncolorable degree-feasible configuration,
\item[(B)] $\cK$ is not constructible, and
\item[(C)] $|G|$ is minimum subject to (A) and (B).
\end{itemize}
By Proposition~\ref{prop_uncolorable-facts}(a) we have
\begin{align}\label{eq_degree-sum}
f(X_v) = d_G(v) \mbox{ for all } v \in V(G).
\end{align}
Clearly, $|G| \geq 2$, as for $|G|=1$ we have $V(G)=\{v\}$ and $f(x)=0$ for all $x \in X_v$ implying that $(G,X,H,f)$ is an M-configuration and hence constructible, a contradiction to (B). We reach a contradiction via a sequence of twelve claims.

\begin{claim} \label{claim_block}
$G$ is a block.
\end{claim}
\begin{proof2}
Otherwise, $G$ is the union of two connected graphs $G^1$ and $G^2$ such that $V(G^1)\cap V(G^2)=\{v^*\}$ and $|G^i| < |G|$ for $i \in \{1,2\}$. For $i\in\{1,2\}$, let $(X^i,H^i)=(X,H)/V(G^i)$ which is a cover of $G^i$. We now define a vertex function $f^i$ of $H^i$ as follows. By Proposition~\ref{prop_uncolorable-facts}(c), $(X,H)$ has a partial transversal $T$ such that $\dom(T:G)=V(G-v^*)$ and $H[T]$ is strictly $f$-degenerate. Let $T_1=T \cap V(G^1)$ and let $T_2= T \cap V(G^2)$. Then $H[T]$ is the disjoint union of $H[T_1]$ and $H[T_2]$. Then, using Proposition~\ref{prop_uncolorable-facts}(c), we obtain that
$$f(x)=d_{H[T \cup \{x\}]}(x)=d_{H[T_1 \cup \{x\}]}(x) + d_{H[T_2 \cup \{x\}]}(x)$$
for all $x \in X_{v^*}$, and we set $f^i(x)=d_{H[T_i \cup \{x\}]}(x)$ for $i \in \{1,2\}$ and $x \in X_{v^*}$. For a vertex $v$ of $G^i-v^*$, let $f^i(x)=f(x)$ for all $x\in X_v$. Clearly, $\cK^i=(G^i,X^i,H^i,f^i)$ is a configuration for $i\in \{1,2\}$.

First, we claim that $\cK^i$ is uncolorable for $i\in \{1,2\}$. For otherwise, by symmetry, we may assume that $\cK^1$ is colorable. Then, there is a transversal $T^1$ of $(X^1,H^1)$ such that $H^1[T^1]$ is strictly $f^1$-degenerate. Clearly, $T=T^1 \cup T_2$ is a transversal of $(X,H)$, and we claim that $H[T]$ is strictly $f$-degenerate. Otherwise, there is a subgraph $\tilde{H}$ of $H[T]$ with $d_{\tilde{H}}(x) \geq f(x)$ for all $x \in V(\tilde{H})$. Since $H[T_2]$ is strictly $f$-degenerate, $\tilde{H}$ contains vertices of $T^1$. Since $H[T^1]$ is strictly $f^1$-degenerate, there is a vertex $y\in V(\tilde{H})\cap T^1$ such that
$d_{\tilde{H}-T_2}(y) < f^1(y)$. If $y \not \in X_{v^*}$, then
$$f(y)\leq d_{\tilde{H}}(y)= d_{\tilde{H}-T_2}(y) < f^1(y) = f(y),$$
which is impossible. If $y \in X_{v^*}$, then $f^2(y)=d_{H[T_2 \cup \{y\}]}(y)$ and we obtain that %
$$f(y)\leq d_{\tilde{H}}(y)= d_{\tilde{H}-T_2}(y) + d_{\tilde{H}[T_2\cup \{y\}]}(y) < f^1(y) + f^2(y) = f(y),$$
which is impossible, too. Hence, $H[T]$ is strictly $f$-degenerate and so $\cK$ is colorable, a contradiction to (A). This shows that $\cK^i$ is uncolorable for $i\in \{1,2\}$, as claimed.

Next, we claim that $\cK^i$ is degree-feasible for $i\in \{1,2\}$. By \eqref{eq_degree-sum} and the definition of $f^i$, we obtain that $f^i(X_v)=f(X_v)=d_G(v)=d_{G^i}(v)$
for all $v \in V(G^i-v^*)$. Moreover, we have
\begin{align}\label{eq_degree_feasible}
d_G(v^*)  = f(X_{v^*}) =f^1(X_{v^*}) + f^2(X_{v^*})=d_{G^1}(v^*) + d_{G^2}(v^*).
\end{align}
Since $\cK^i$ is uncolorable, it follows from Proposition~\ref{prop_uncolorable-facts}(a) that
$f^i(X_{v^*})\leq d_{G^i}(v^*)$
for $i\in \{1,2\}$. By \eqref{eq_degree_feasible}, this implies that
$f^i(X_{v^*})= d_{G^i}(v^*)$
for $i\in \{1,2\}$. Consequently, $\cK^i$ is degree-feasible for $i\in \{1,2\}$.

Since $\cK=(G,X,H,f)$ is a minimal counter-example and $|G^i|<|G|$ for $i\in \{1,2\}$, we conclude that $\cK^i=(G^i,X^i,H^i,f^i)$ is a constructible configuration, and so $\cK$ is obtained from the constructible configurations $\cK^1$ and $\cK^2$ by merging two vertices to $v^*$. Hence, $\cK$ is a constructible configuration, a contradiction to (B).
\end{proof2}

\begin{claim} \label{claim_reduction}
Let $v\in V(G)$ and let $x_v\in X_v$ such that $f(x_v)>0$. Then the configuration $\cK'=\cK/(v,x_v)$ is constructible. If $f'$ is the function of $\cK'$, then for every vertex $u$ of $G-v$, we have
\begin{itemize}
\item[\upshape(a)] $f'(y)=f(y)-\mu_H(y,x_v)\geq 0$ for all $y\in X_u$, and
\item[\upshape(b)] $\mu_G(u,v)=\sum_{y\in X_u}\mu_H(y,x_v).$
\end{itemize}
\end{claim}
\begin{proof2}
From (A) and Proposition~\ref{prop_reduction} it follows that $\cK'=\cK/(v,x_v)$ is an uncolorable  degree-feasible configuration. Since $G$ is a block (by Claim~\ref{claim_block}), $G-v$ is connected. By (B) and (C), this implies that $\cK'$ is a constructible configuration. Furthermore, we have $f(y)\geq \mu_H(y,x_v)$ (by Proposition~\ref{prop_uncolorable-facts}(b)) and $f'(y)=\max\{0,f(y)-\mu_H(y,x_v)\}$, which yields (a). Statement (b) also follows from Proposition~\ref{prop_uncolorable-facts}(b).
\end{proof2}

Let $uv$ be an edge of $G^o$, let $X\subseteq X_v$ and $X'\subseteq X_u$. We call $(X,X')$ a \DF{complete $uv$-pair} of \DF{type} $(t,s)$ if $H(X,X')$ is a $tK_{s,s}$ and $\mu_H(u,v)=ts$. Note that if $(X,Y)$ is a complete $uv$-pair of type $(t,s)$ and $U=X\cup Y$, then $G'=G[\{u,v\}]$ is a $(ts)K_2$, and $(X,H)/U$ is an $s$-inflation of a $G'$-cover of $G'$. Let
$$CP(G)=\set{uv\in E(G^o)}{(X_u,X_v) \mbox{ is a complete $uv$-pair}}$$

\smallskip

Let $U=\su(f)$ be the support of $f$, and let $f^*$ be the restriction of $f$ to $U$.
Then it follows from Claim~\ref{claim_reduction} that
the cover $(X',H')=(X,H)/U$ is $G$-saturated. Furthermore, $(G,H',X',f^*)$ is an uncolorable  degree-feasible configuration which is not constructible. Hence $(G,X',H',f^*)$ is also a smallest counterexample and we may assume that $U=V(H)$. As an immediate consequence of Claim~\ref{claim_reduction}, we then obtain the following result.

\begin{claim} \label{claim_reductionB}
For the configuration $\cK=(G,X,H,f)$ the following statements hold:
\begin{itemize}
\item[{\rm (a)}] For every edge $uv$ of $G^o$, the graph $H(X_u,X_v)$ is a bipartite graph with parts $X_u$ and $X_v$ that is regular of degree $\mu_G(u,v)$.
\item[{\rm (b)}] If $v\in V(G)$ and $x\in X_v$, then $\cK/(v,x)$ is a constructible configuration.
\end{itemize}
\end{claim}

Since $G$ is connected, it follows from Claim~\ref{claim_reductionB}(a) that $\cK$ is $r$-uniform for an integer $r\geq 1$. If $r=1$, then $X_v=\{x_v\}$ for all $v\in V(G)$ and $\mu_H(x_v,x_u)=\mu_G(u,v)$ for every pair $u, v$ of distinct vertices of $G$ implying that $(X,H)$ is a $G$-cover of $G$ and $f(x_v)=d_G(v)$ for all $v\in V(H)$. Hence $\cK$ is an M-configuration. This contradiction to (B) shows that $r\geq 2$.

Let $v$ be an arbitrary vertex of $G$, and let $x$ be an arbitrary vertex of $X_v$. Then define $\cK_x=\cK/(v,x)$. By Claim~\ref{claim_reductionB}, $\cK_x$ is a constructible configuration and we denote by $f_x$ the function of $\cK_x$. Note that $G'=G-v$ is a connected graph and $\cK_x=(G',X^{G'},H^{G'},f_x)$. Since $\cK_x$ is constructible, we can use the block decomposition of $\cK_x$ described in Proposition~\ref{prop_construtible-blocks}. If $B\in \cB(G')$, then we denote by $f_x^B$ the $B$-part of the function $f_x$. One important consequence of Proposition~\ref{prop_construtible-blocks} and Claim~\ref{claim_reduction}(a)
is the following:
\begin{align}\label{eq_reduction}
f_x^B(y)=f_x(y)=f(y)-\mu_H(y,x) \mbox{ whenever } y\in X_u \mbox{ and } u\in V(B)\sm S(G').
\end{align}

\begin{claim}  \label{claim_neighborhood}
Let $v \in V(G)$ be an arbitrary vertex, let $G'=G-v$, let $B\in \cB(G')$, and let $u, u'$ be two distinct vertices of $V(B)\sm S(G')$. If $u\in N_G(v)$ and $u'\not\in N_G(v)$, then $uv\in CP(G)$.
\end{claim}
\begin{proof2}
Suppose that $(X_u,X_v)$ is not a complete pair. Let $m=\mu_H(u,v)$. By Claim~\ref{claim_reductionB}(a), $H(X_u,X_v)$ is an $m$-regular bipartite graph with parts $X_u$ and $X_v$. Consequently, there is a vertex $y\in X_u$ and two distinct vertices $x_1, x_2\in X_v$ such that $\mu_H(x_1,y)\not=\mu_H(x_2,y)$. For $i \in \{1,2\}$, let $f_i=f_{x_i}^B$. Then $\cK_{x_i}^B=(B,X^B,H^B,f_i)$ is an M-, or K-, or C-configuration ($i\in \{1,2\}$). Since $u'\not\in N_H(v)$, we have $f_1(z)=f_2(z)=f(z)>0$ for all $z\in X_{u'}$ (by \eqref{eq_reduction}). Consequently, $\su(f_i)=V(H^B)$ for $i\in \{1,2\}$. Since $\mu_H(x_1,y)\not=\mu_H(x_2,y)$, we have $f_1(y)\not=f_2(y)$ (by \eqref{eq_reduction}).
First assume that $\cK_{x_1}^B$ is an M-configuration. Then $\cK_{x_2}^B$ is an $M$-configuration, too (Proposition~\ref{prop_construtible-blocksB}), implying that $f_1(y)=d_B(u)/r=f_2(y)$, a contradiction. Otherwise, $\cK_{x_1}^B$ is a K-configuration with at least two layers, or a C-configuration where the underlying cycle has at least four vertices. This also leads to $f_1(y)=f_2(y)$, a contradiction. This completes the proof.
\end{proof2}

\begin{claim} \label{claim_CP1}
We have $CP(G)\not=E(G^o)$.
\end{claim}
\begin{proof2}
Suppose that $CP(G)=E(G^o)$. Then $(X,H)$ is an $r$-inflation of a $G$-cover of $G$. Our aim is to show that $\cK$ is an M-configuration, a contradiction to (B). To this end, it suffices to show that $f(y)=d_G(u)/r$ whenever $y\in X_u$ and $u\in V(G)$. So let $u\in V(G)$. Since $|G|\geq 2$, there is a vertex $v\in N_G(u)$ and a vertex $x\in X_v$. Then $vu\in CP(G)$, and so $H(X_u,X_v)$ is a $tK_{r,r}$ with $tr=\mu_G(u,v)$, which leads to $f(y)=f_x(y)+t$ for all $y\in X_u$. Note that $G'=G-v$ is a connected graph and $\cK_x$ is constructible. If $u$ is the only vertex of $G'$, then $f_x(y)=0$ for all $y\in X_u$, which leads to $f(y)=t=d_G(u)/r$ as claimed. It remains to consider the case that $|G'|\geq 2$. Let $B\in \cB(G')$ be an arbitrary block. Then $|B|\geq 2$ and hence there is an edge $ww'\in E(B^o)$. Since $ww'\in CP(G)$, we obtain that $H^B(X_w,X_{w'})=H(X_w,X_{w'})$ is a $t'K_{r,r}$ with $t'r=\mu_G(w,w')$. Hence $\cK_x^B$ is an M-configuration (Proposition~\ref{prop_construtible-blocksB}), and, therefore, $f_x^B(z)=d_B(z)/r$ whenever $z\in X_w$ and $w\in V(B)$. Using Proposition~\ref{prop_construtible-blocks} for $\cK_x$, for every vertex $y\in X_u$ obtain that
$$f_x(y)=\sum_{B\in \cB_u(G')}f_x^B(y)=\frac{1}{r}\sum_{B\in \cB_u(G')}d_B(u)=\frac{d_G'(u)}{r},$$
which implies that $f(y)=f_x(y)+t=\tfrac{1}{r}(d_G'(u)+\mu_G(u,v))=\tfrac{1}{r}d_G(u).$ This completes the proof.
\end{proof2}

Assume that $G$ has only two vertices, say $u$ and $v$. For any vertex $x\in X_v$, the configuration $\cK_x$ is constructible and so $f_x(y)=0$ for all $y\in X_u$, which implies that $f(y)=f_x(y)+\mu_H(y,x)=\mu_H(y,x)$ for $y \in X_u$. By symmetry, we also obtain for any vertex $y\in X_u$ that $\mu_H(y,x)=f(x)>0$ for all $x \in X_v$. Furthermore, $f(X_v)=f(X_u)=d_G(u)=d_G(v)$. This implies that $H(X_u,X_v)$ is a $tK_{r,r}$ and $f(x)=t$ for all $x\in V(H)$. Hence $\cK$ is an M-configuration. This contradiction to (B) shows that $|G|\geq 3$.

\begin{claim} \label{claim_CP2}
We have $CP(G)=\ems$.
\end{claim}
\begin{proof2}
Suppose that $CP(G)\not=\ems$. By Claim~\ref{claim_CP1}, there is an edge $e=uw\in E(G^o)\sm CP(G)$. Since $G$ is a block, there is a cycle in $G^o$ containing the edge $e$ and an edge belonging to $CP(G)$. Let $C$ be a shortest such cycle. Then $C$ is an induced cycle of $G^o$. First assume that there is a vertex $v\in V(G)\sm V(C)$. Then $\cK_x$ is an constructible configuration where $x\in X_v$. Clearly, there is a block $B$ of $G-v$ containing $V(C)$. Then $B^o$ contains $uv$ and an edge of $CP(G)$, which implies, by Proposition~\ref{prop_construtible-blocksB}, that $uv\in CP(G)$, a contradiction. It remains to consider the case that $G^o=C$. Then there are three vertices $v,w,w'$ such that $vw\in E(C)\sm CP(G)$ and $ww'\in E(C)\cap CP(G)$. Then $B=G[\{w,w'\}]$ is a block of $G'=G-v$. Let $x\in X_v$ be an arbitrary vertex. Clearly, $\cK_x^B$ is an M-configuration. Since $ww'\in CP(G)$, we have $\su(f_x^B)=X_w \cup X_{w'}$, which implies that $f_x^B(y)=t$ for all $y\in X_w$. Since $w\not\in S(G')$, we obtain, for $y\in X_w$, that $f_x(y)=f_x^B(y)=t$ and $f(y)=t+\mu_H(y,x)$ (by \eqref{eq_reduction}).
Since $x$ was chosen arbitrarily in $X_v$ and $H(X_v,X_w)$ is a regular bipartite graph, we obtain that $H(X_v,X_w)$ is a $t'K_{r,r}$. Hence $vw\in CP(G)$, a contradiction.
\end{proof2}

Combining Claim~\ref{claim_CP2} and Claim~\ref{claim_neighborhood}, we obtain the following result.

\begin{claim}  \label{claim_neighborhoodB}
Let $v \in V(G)$ be an arbitrary vertex, let $G'=G-v$, let $B\in \cB(G')$, and let $u, u'$ be two distinct vertices of $V(B)\sm S(G')$.
Then, either $\{u,u'\} \subseteq N_G(v)$ or $\{u,u'\} \cap N_G(v) = \varnothing$.
\end{claim}

\begin{claim} \label{claim_G-complete+cycle}
The simple graph $G^o$ is a cycle or a complete graph.
\end{claim}
\begin{proof2}
Suppose that $G^o$ is not a complete graph. Since $G$ is a block and $|G|\geq 3$, we have $2\leq \de(G^o)\leq |G^o|-2=|G|-2$. Let $v\in V(G)$ be a vertex of minimum degree in $G^o$.  If $B=G-v$ is a block, then $B$ contains a vertex $u\in N_G(v)$ and a vertex $u'\not\in N_G(v)$, contradicting Claim~\ref{claim_neighborhoodB}. So $G'=G-v$ is not a block and there are at least two end-blocks of $G'$. Let $B$ be an arbitrary end-block of $G'$. By the choice of $v$, we obtain that $|B|\geq \de(G^o)$. Since $B$ is an end-block of $G'$, there is exactly one vertex $u\in V(B)\cap S(G')$. Since $G$ is a block, $v$ has in $G$ a neighbor belonging to $B-u$. By Claim~\ref{claim_neighborhoodB} this implies that $V(B-u)\subseteq N_G(v)$. Since $G$ has at least two end-blocks, this leads to $\de(G^o)\geq 2(\de(G^o)-1)$ and hence to $\de(G^o)=2$. Then $|B|=2$ and the vertex $w\in V(B-u)$ has degree $\de(G^o)$ in $G^o$. Furthermore $G'=G-v$ has exactly two end-blocks. If we repeat this argument with $w$, we obtain that $G^o$ is a cycle.
\end{proof2}

Let $uw$ be an arbitrary edge of $G^o$ and let $m=\mu_H(u,w)$. Then $m>0$ and the bipartite graph $H(X_u,X_w)$ is regular of degree $m$ (Claim~\ref{claim_reductionB}(a)). A component of $H(X_u,X_w)$ is called a \DF{$uw$-part} of $H$. For a $uw$-part $H'$, let $X_u(H')=X_u\cap V(H')$ and $X_w(H')=X_w \cap V(H')$. If $(X_u(H'),X_w(H'))$ is a complete $uw$-pair, then $H'$ is a $tK_{s,s}$ with $m=ts$; in this case we say that $H'$ is a \DF{full $uv$-part} of $H$ of \DF{type} $(t,s)$. Note that if $H_1$ and $H_2$ are full $uw$-parts of $H$, then $V(H_1)=V(H_2)$ or $V(H_1)\cap V(H_2)=\ems$, and in the later case $E_H(V(H_1),V(H_2))=\ems$.

\begin{claim} \label{claim_G-component}
If $uw$ is an edge of $G^o$, then $H(X_u,X_v)$ is the disjoint union of $p$ full $uw$-parts with $p\geq 2$.
\end{claim}
\begin{proof2}
Let $uw$ be an arbitrary edge of $G^o$. Since $|G|\geq 3$, there is a vertex $v\in N_G(u)\sm \{w\}$.
Then $G'=G-v$ is connected and there is a block $B\in \cB(G')$ containing $u$ and $w$. Since
$G^o$ is a cycle or a complete graph (by Claim~\ref{claim_G-complete+cycle}), it follows that $u\not\in S(G')$. Let $x\in X_v$ be an arbitrary vertex. Then $\cK_x$ is a constructible configuration. Since $CP(G)=\ems$, it then follows from Proposition~\ref{prop_construtible-blocksB} that $H(X_u,X_w)=H^B(X_u,X_w)$ has at least two components. Since $uw$ was chosen arbitrarily,
the same holds for the bipartite graph $H(X_v,X_u)$. Now let $y$ be an arbitrary vertex of $X_u$. Then for $x$ we can choose a vertex in $X_v$ that is no neighbor of $y$ in $H$. By \eqref{eq_reduction}, this implies that $f_x^B(y)=f(y)>0$. Then $y$ belongs to a full $uw$-part of $H$, since $\cK_x^B$ is an M-, or K-, or C-configuration, $y\in \su(f_x^B)$, and $H(X_u,X_w)=H^B(X_u,X_w)$. This proves the claim.
\end{proof2}

\begin{claim} \label{claim_layer}
Let $vu$ and $uw$ be two distinct edges of $G^o$, let $(X,Y)$ be a complete $vu$-pair, and $(Z,W)$ be a complete $uw$-pair. Suppose that $Y\cap Z\not=\ems$. Then $Y=Z$ and, moreover, the following statements hold:
\begin{itemize}
\item[\upshape(a)] If $G^o$ is a complete graph of order $n\geq 4$, then $(X,W)$ is a complete $vw$-pair.
\item[\upshape(b)] If $G^o$ is a cycle, then there are exactly two complete $vu$ pairs.
\end{itemize}
\end{claim}
\begin{proof2}
First assume that $G^o$ is a complete graph of order $n\geq 4$. Then there is a vertex $y\in Y\cap Z$ and a vertex $v'\in V(G)\sm \{u,v,w\}$. By Claim~\ref{claim_G-component}, there is a vertex $x\in X_{v'}$ such that $\mu_H(x,y)=0$. Then we have $f_x(y)=f(y)>0$ (by Claim~\ref{claim_reduction}(a)). Since $G^o$ is a complete graph, $\cK_x$ is an M-configuration or a K-configuration and $y$ belongs to a layer $U$ of $\cK_x$. Then $X,Y,W,$ and $Z$ are all contained in $U$, which implies that $Y=Z$ and $(X,W)$ is a complete $vw$-pair. So we are done.

Now assume that $G^o$ is a cycle. Then $G^o-v$ is a path, $B=G[\{u,w\}]$ is an end-block of $G'=G-v$ and $u\not\in S(G')$. Let $x\in X_v$ be an arbitrary vertex, and let $U_x=\su(f_x^B)$ and $U_x^c=\suc(f_x^B)$. Then $\cK_x^B=(B,X^B,H^B,f_x^B)$ is an M-configuration implying that $H[U_x]$ is a full $uw$-part of $H$. By Claim~\ref{claim_G-component}, $U_x\not=\ems$. For $x, x'\in X_v$, we have $U_x=U_{x'}$ or $U_x\cap U_{x'}=\ems$. Now let $X'=\set{x\in X_v}{H[U_x]=H(Z,W)}$. By \eqref{eq_reduction}, we obtain that for $y\in X_u$ and $x\in X_v$ we have
$$f_x^B(y)=f_x(y)=f(y)-\mu_H(x,y)$$
Suppose that $H(Z,W)$ is a $tK_{s,s}$. Let $Z'=X_u\sm Z$. If $x\in X'$, then $Z'\subseteq \suc(f_x^B)$ which yields that $\mu_H(x,y)=f(y)>0$ for all $y\in Z'$. Consequently, $Z'\subseteq N_H(x)$ for all $x\in X'$. Now let $x\in X_v$ be a vertex such that there is an $y\in Z$ with $\mu_H(x,y)=0$. Then $f_x(y)=f(y)>0$ which implies that $Z\subseteq \su(f_x)$ and so $x\in X'$. Consequently, $Z \subseteq N_H(x)$ for all $x\in X_v\sm X'$. From Claim~\ref{claim_G-component} it then follows that $(X',Z')$ and $(X_v\sm X',Z)$ are the only complete $vu$-pairs. This implies (b).
\end{proof2}

\begin{claim} \label{claim_onlycycle}
$G^o$ is a complete graph of order $n\geq 4$.
\end{claim}
\begin{proof2}
Suppose this is false. Then, by Claim~\ref{claim_G-complete+cycle}, $G^o$ is a cycle $C$ and $n=|C|\geq 3$. Let $v$ be an arbitrary vertex of $C$, let $u$ and $w$ be the two neighbors of $v$ in $C$, and let $u'$ be the neighbor of $u$ in $C$ different from $w$. Then it follows from Claim~\ref{claim_layer}(b), that $V(H)$ has a partition into two sets, say $U_1$ and $U_2$, such that, for every edge $v'w'$ of the path $C-uu'$, $(X_{v'}(U_i),X_{w'}(U_i))$ is a complete $v'w'$-pair of type $(t_i(v'w'),s_i(v'w'))$ for $i\in \{1,2\}$. Then, by Claim~\ref{claim_layer}(b), either $(X_u(U_1),X_{u'}(U_1))$ or $(X_u(U_1),X_{u'}(U_2)$ is a complete $uu'$ pair.

\ncase{1}{$(X_u(U_1),X_{u'}(U_1))$ is a complete $uu'$ pair.} Then $(X_u(U_2),X_{u'}(U_2))$ is a complete $uu'$ pair, too, and $(X,H)/U_i$ is an $s_i$-inflation of $G$ ($i\in \{1,2\}$. Furthermore, we obtain that $n$ is odd. For otherwise, $(X,H)$ has an independent transversal and so $\cK$ is colorable, a contradiction to (A). Choose two vertices $y\in X_u(U_1)$ and $y'\in X_w(U_2)$. Since $n$ is odd, there is a partial transversal $T$ of $(X,H)$ such that $\dom(T:G)=V(G-v)$, $y,y'\in T$, and $T$ is an independent set of $H$, which implies that $H[T]$ is strictly $f$-degenerate. By Proposition~\ref{prop_uncolorable-facts}(c), we obtain that $f(x)=d_{H[T\cup \{x\}]}(x)$ for all $x\in X_v$. Consequently, $f(x)=\mu_H(x,y)=t_1(vu)$ for all $x\in X_v(U_1)$ and $f(x)=\mu_H(x,y')=t_2(vw)$ for all $x\in X_v(U_2)$. Now we can choose two vertices $y\in X_u(U_2)$ and $y'\in X_w(U_1)$ to show that $f(x)=\mu_H(x,y)=t_1(vw)$ for all $x\in X_v(U_1)$ and $f(x)=\mu_H(x,y')=t_2(vu)$ for all $x\in X_v(U_2)$. This implies that $t_i(vu)=t_i(vw)$ for $i\in \{1,2\}$. Since $v$ was chosen arbitrarily, it then follows that $G=tC_n$ and $f(x)=t_i$ for all $x\in U_i$ with $t_i s_i=t$. Hence $\cK$ is an odd C-configuration, a contradiction to (B).

\ncase{2}{$(X_u(U_1),X_{u'}(U_2))$ is a complete $uu'$ pair.} Then $(X_u(U_2),X_{u'}(U_1))$ is a complete $uu'$ pair, too, implying that $|X_u(U_i))|=s$ for $i \in \{1,2\}$ and all $u\in V(G)$, where $s\in \nat$. This implies that $t_1(v'w')=t_2(v'w')=\mu_G(v',w')/s$ for all edges $v'w'$ of $C-uu'$. Then we obtain that $n$ is even,
For otherwise, $(X,H)$ has an independent transversal and so $\cK$ is colorable, a contradiction to (A).
Now we may argue similarity as in the first case to show that $G=tC_n$ and $f(x)=t/s$ for all $x\in V(H)$. Hence $\cK$ is an even C-configuration, a contradiction to (B).
\end{proof2}

By the above claim, $G^o$ is a complete graph of order $n\geq 4$. Then it follows from Claim~\ref{claim_layer}(a) that $V(H)$ has a partition into $p$ sets, say $U^1, U^2, \ldots, U^p$ such that $(X,H)/U^i$ is an $s_i$-inflation of $G$ for $i\in [1,p]$ and $p\geq 2$ (by Claim~\ref{claim_G-component}). Then for every $i\in [1,p]$ and every edge $vu$ of $G^o$, there is an integer $t_i(uv)$ such that $H(X_u(U^i),X_v(U^i))$ is a $t_i(uv)K_{s,s}$ and $\mu_G(u,v)=t_i(uv)s_i$. Furthermore, $E_H(U^i,V(H)\sm U^i)=\ems$. Our aim is to show that $\cK$ is a K-configuration. This final contradiction then completes the proof of Theorem~\ref{theorem:main2}.

\smallskip

Let $x\in V(H)$ be an arbitrary vertex. Then $x\in X_u$ for exactly one vertex $u\in V(G)$. For $i\in [1,p]$, we define $U^i_x=U^i\sm X_u$. Then $(X,H)/U^i_x$ is an $s_i$-inflation of $G-v$. Since $\cK_x$ is constructible (by Claim~\ref{claim_reductionB}(b)) and $G^o-v$ is a complete graph, $\cK_x$ is an M-configuration or a K-configuration. For the function $f_x$ of $\cK_x$ we obtain that $f_x(y)=f(y)-\mu_H(x,y)$ for all $y\in V(H)\sm X_u$ (by Claim~\ref{claim_reduction}(a)). Consequently, if $x\in U^i$, $j\in [1,p]\sm \{i\}$ and $G'=G-u$, then the following statements hold:
\begin{itemize}
\item[\upshape(1)] $f_x(y)=f(y)$ for all $y\in U^j_x$, and $U^j_x$ is a layer of $\cK_x$. Furthermore, if $v\in V(G')$ and $y\in X_v(U^j)$, then $f_x(y)=d_{G'}(v)/s_j$.
\item[\upshape(2)] $f_x(y)=f(y)-t_i(uv)$ for all $v\in V(G')$ and $y\in X_v(U^i)$.
\end{itemize}

\begin{claim} \label{claim_mono}
No configuration $\cK_x$ with $x\in V(H)$ is an M-configuration.
\end{claim}
\begin{proof2}
Suppose, this is false. Then $\cK_{x}$ is an M-configuration for a vertex $x\in V(H)$, say $x\in X_u$  for $u\in V(G)$. Let $G'=G-u$, let $U=\su(f_{x})$ and $U^c=\suc(f_{x})$. Then $(X,H)/U$ is an inflation of a $G'$-cover and $f_{x}(y)=0$ for all $y\in U^c$. By (1), this implies that $p=2$ and either $U=U^1_x$ or $U=U^2_x$. By symmetry we may assume that $U=U^1_x$ and hence $x\in U_2$. Since $U^1_x$ is the only layer of the M-configuration $\cK_x$, we obtain from (1) and (2) that
\begin{itemize}
\item[\upshape(3)] $f(y)=f_x(y)=d_{G'}(v)/s_1$, provided that $v\in V(G')$ and $y\in X_v(U^1)$, and
\item[\upshape(4)] $f(y)=t_2(uv)$, provided that $v\in V(G')$ and $y\in X_v(U^2)$.
\end{itemize}
Let $T$ be an arbitrary transversal of $(X,H)/U_x^1$. Then, for every vertex $v\in V(G')$, denote by $y(v)$ the unique vertex in $T\cap X_v(U^1)$. Furthermore, since $G'^o$ is a complete graph, we obtain that $d_{H[T]}(y)=f_x(y)=f(y)$ for all $y\in T$, where the second equation follows from (1). Now let $v$ be an arbitrary vertex of $G'$, and let $y\in X_v(U^2)$ be an arbitrary vertex.  Replace in $T$ the vertex $y(v)$ by $y$ and denote the resulting set by $T'$. Then $T'$ is a partial transversal of $(X,H)$. Since $E_H(U^1,V(H)\sm U^1)=\ems$, we obtain that $y$ is an isolated vertex in $H[T']$ and, therefore, $H[T']$ is strictly $f$-degenerate. Since $\dom(T':G)=V(G')$, it then follows from Proposition~\ref{prop_uncolorable-facts} that for all $x'\in X_u(U^2)$ we have
$f(x')=d_{H[T'\cup \{x'\}]}(x')=\mu_H(x',y)=t_2(uv)=f(y),$
where the last equality follows from (3).
Since $(v,y)$ was chosen arbitrarily with $v\in V(G')$ and $y\in X_v(U^2)$, we obtain that there is an integer $t_2$ such that $f(z)=t_2$ for all $z\in U^2$ and $t_i(uv)=t_2$ for all $v\in V(G')$. Let $vv'$ be an arbitrary edge of $G^o-u$. Then Claim~\ref{claim_reduction}(a) implies that $t_2\geq t_2(vv')$. We claim that equality holds. For otherwise, $t_2> t_2(vv')$, and we choose two vertices $y\in X_v(U^2)$ and $y'\in X_{v'}(U^2)$, and let $T'$ be the set obtained from $T$ by replacing $y(v), y(v')$ by $y, y'$. Then $T'$ is a partial transversal of $(X,H)$ such that $\dom(T':G)=V(G')$, $E_H(T'\sm \{y,y'\},\{y,y'\})=\ems$, and $d_{H[T']}(y)=d_{H[T']}(y')=t_2(vv')<t_2=f(y)=f(y').$ Consequently, $H[T']$ is strictly $f$-degenerate and Proposition~\ref{prop_uncolorable-facts} implies that, for the vertex $x\in X_u$, we have
$$0<t_2=f(x)=d_{H[T'\cup \{x\}]}(x)=\mu_H(x,y)+\mu_H(x,y')=f(y)+f(y')=2t_2,$$
which is impossible. This proves the claim that $t_2(vv')=t_2$. Consequently, $G=tK_n$ with $t=s_2t_2$, $(X,H)/U_2$ is an $s_2$-inflation of $G$, and $f(z)=t_2$ for all $z\in U^2$. Then $G'=t K_{n-1}$ and it follows from (3) that

\begin{align}\label{equation_mono4}
f(y)=t(n-2)/s_1 \mbox{ for all } y\in U_x^1.
\end{align}
Let $t_1=t/s_1$. Let $x'\in X_u(U^1)$ be an arbitrary vertex. There is a partial transversal $T$ of $(X,H)$ such that $\dom(T:G)=V(G')$ and $|T\cap U^1|=|G'|-2=n-2$. If $y'$ is the only vertex of $T$ belonging to $U^2$, then $y'$ is an isolated vertex of $H[T]$. For every vertex $y\in T \cap U^1$, we have $d_{H[T]}(y)=t_1(n-3)=t(n-3)/s_1<f(y)$ (by \eqref{equation_mono4}). Hence $H[T]$ is strictly $f$-degenerate, and Proposition~\ref{prop_uncolorable-facts}(c) then yields that $f(x')=d_{H[T\cup \{x'\}]}(x')=t_1(n-2)$. Since $(X,H)/U^1$ is an $s_1$-inflation of $G$ and $G=tK_n$, this implies that $\cK$ is a K-configuration, where $U_1$ is a layer of type $n_1=n-2$, and $U_2$ is a layer of type $n_2=1$. This contradiction completes the proof.
\end{proof2}

Since $G^o$ is a complete graph of order $n\geq 4$, Claim~\ref{claim_mono} implies that $\cK_x$ is a K-configuration for all $x\in V(H)$ and, therefore, $G-v$ is a $t_vK_n$ for all $v\in V(G)$. Since $n\geq 4$, this implies that $G=tK_n$ with $t\in \nat$. For $i\in [1,n]$, $(X,H)/U^i$ is an $s_i$-inflation of $G$, which implies that there is a $t_i\in \nat$ such that $t=s_it_i$. Now we claim that the function $f$ of $\cK$ restricted to $U^j$ is constant. So let $y, y'$ be two vertices of $U^j$. Then there is a vertex $u\in V(G)$ such that neither $y$ nor $y'$ belongs to $X_u$. Since $p\geq 2$, there is a vertex $x\in X_u(U^i)$ with $i\not=j$. By (1), this implies that $U^j_x$ is a layer of the M-configuration $\cK_x$. Then $f_x(y)=f_x(y')$ and, by (1), $f(y)=f_x(y)=f_x(y')=f(y')$. This proves the claim.

\smallskip

Now, let $u$ be an arbitrary vertex of $G$. By Proposition~\ref{prop_uncolorable-facts}(c), there is a partial transversal of $(X,H)$ such that $\dom(T:G)=V(G')$ and $H[T]$ is strictly $f$-degenerate. For $i\in [1,p]$, let $n_i=|T\cap U^i|$ implying that $n_1+n_2+\cdots +n_p=n-1$. By Proposition~\ref{prop_uncolorable-facts}(c), $f(x)=d_{H[T\cup \{x\}]}(x)$ for all $x\in X_u$, which implies that $f(x)=t_in_i=tn_i/s_i$ when $x\in X_u(U^i)$ ($i\in [1,p]$). Consequently, $f(z)=tn_i/s_i$ for all $z\in U^i$ ($i\in [1,p]$), and so $\cK$ is a K-configuration, where $U^i$ is a layer of order $n_i>0$. This contradiction completes the proof of Theorem~\ref{theorem:main2}
\end{proof}

\pffbf{Theorem~\ref{theorem:lowvertexA1}}
{Let $\cP$ be a reliable graph property with $d(\cP)=r$, let $G$ be a graph, let $(X,H)$ be a $\cP$-critical cover of $G$, and let $B$ be an arbitrary block of the low vertex subgraph
$G[V(G,X,H,\cP)]$, and let $G'=G-V(B)$. Since $(X,H)$ is a
$\cP$-critical cover of $G$, there is a partial transversal $T$ of $(X,H)$ such that $\dom_G(T)=V(G')$ and $H[T]\in \cP$. For a vertex $u\in V(B)$ and a color $x\in X_u$, let
$H_u=H[T \cup \{x\}]$ and $d_{u,x}=d_{H_u}(x)$. Let $U$ be the union of the sets $X_u$ with $u\in V(B)$, and let $(X',H')=(X,H)/U$. Furthermore, define a vertex function $f$ for $H$ by
$$f(x)=\max \{0, r-d_{u,x}\}$$
whenever $u\in V(B)$ and $x\in X_u$. Note that $X'=X^B$ and $H'=H^B$. First assume that $(X',H')$ has a transversal $T'$ such that $H'[T']$ is strictly $f$-degenerate. Note that this implies that $f(x)>0$ for all $x\in T'$. Furthermore, $T'\cup T$ is a transversal of $(X,H)$, and hence $H[T'\cup T]\not\in \cP$. From Proposition~\ref{prop:smooth}(c) it then follows that there is a set $T_1\subseteq T' \cup T$ such that $H[T_1]\in \CR(\cP)$. Then Proposition~\ref{prop:smooth}(e) implies that $\de(H[T_1])\geq r$. Since $H[T]\in \cP$, we have $T_1\cap T'\not=\ems$, and so $H[T_1\cap T']$ is a non-empty induced subgraph of $H'[T']=H[T']$. Since $H'[T']$ is strictly $f$-degenerate, $\tilde{H}=H[T_1\cap T']$ contains a vertex $x$ with $d_{\tilde{H}}(x)<f(x)$. Then $x\in X_u$ for some $u\in V(B)$ and $f(x)=r-d_{u,x}$. This leads to $d_{H[T_1]}(x)=d_{\tilde{H}}(x)+d_{u,x}< f(x)+d_{u,x}\leq r$, a contradiction to
$\de(H[T_1])\geq r$.

It remains to consider the case when $(X',H')$ has no transversal that is strictly $f$-degenerate.
Let $u\in V(B)$ be an arbitrary vertex. As $u$ is a low vertex, we have $d_G(u)=r|X_u|$. Furthermore, we have
$$\sum_{x\in X_u}d_{u,x}\leq d_{G-V(B-u)}(u)=d_G(u)-d_B(u),$$
where the first inequality follows from (C2). Then we obtain that
\begin{align}
\label{equation:main:theorem}
\sum_{x\in X_u}f(x)\geq \sum_{x\in X_u}(r-d_{x,u})=r|X_u|-\sum_{x\in X_u}d_{x,u}=d_G(u)-\sum_{x\in X_u}d_{x,u}\geq d_B(u).
\end{align}
Consequently, $\cK=(B,X',H',f)$ is an uncolorable degree-feasible configuration. By Theorem~\ref{main_theorem} it then follows that $\cK$ is a constructible configuration. Since $B$ is a block, $\cK$ is a K-, C-, or M-configuration. In the first two cases, $B$ is a brick, and we are done. It remains to consider the case when $\cK$ is an M-configuration. Then there is a set $U\subseteq V(H')$ such that $(X',H')/U$ is an $s$-inflation of $B$, and for $u\in V(B)$ and $x\in X_u$, we have $s=|X_u(U)|$ and $f(x)=d_B(u)/s$ if $x\in U$ else $f(x)=0$. This implies that $B=sB'$. Consequently, for every vertex $u$ of $B$, we have
$$f(X_u)=d_B(u)=sd_{B'}(u).$$
By \eqref{equation:main:theorem}, this implies that $f(x)=r-d_{x,u}$ whenever $u\in V(B)$ and $x\in X_u$. Hence
$$sd_{B'}(u)=d_B(u)=f(X_u)=f(X_u(U))=\sum_{x\in X_u(U)}(r-d_{u,x})\leq rs,$$
which implies that $\De(B')\leq r$. If $B'\in \cP$, then we are done. If $B'\not\in \cP$, then $B'$ has an induced subgraph $B^*\in \CR(\cP)$ (by Proposition~\ref{prop:smooth}(c)). Then $\de(B^*)\geq d(\cP)=r$, which implies that $B'=B^*$ and $B'$ is $r$-regular. Hence we are done, too. This completes the proof.}

\section{Critical graphs with few edges}
\label{sec:critical+graphs}

Gallai \cite{Gallai63a} established a lower bound for the number of edges possible in a simple graph $G$
being critical with respect to the chromatic number, where the bound is depending on $|G|$ and $\cn(G)$. The proof given by Gallai uses the characterization of the low vertex subgraph that he obtained in \cite{Gallai63a}. We can easily adopt Gallai's proof to establish a Gallai type bound for the number of edges of cover critical simple graphs in general. Our result is an extension of Gallai's result \cite[Satz 4.4]{Gallai63a}. First we need the following result due to Mih\'ok and \v{S}krekovsky \cite[Corollary 4]{MihokS01}; this result is an extension of Gallai's technical lemma \cite[Lemma 4.5]{Gallai63a}.

\begin{theorem}
\label{theorem:Mihok}
Let $p\geq 1$ be an integer. Let $F$ be a non-empty simple graph such that $\De(F)\leq p$ and $\De(B)<p$ for all blocks $B\in \cB(F)$. Then
$$\left( p-1+\frac{2}{p}\right)|F|-2|E(F)|\geq 2.$$
\end{theorem}
\begin{theorem}
\label{theorem:Gallai}
Let $\cP$ be a reliable graph property with $d(\cP)=r$, let $G$ be a simple graph that has a  $\cP$-critical $k$-cover with $k\geq 3$. Then
$$2|E(G)|\geq \left(kr+\frac{kr-2}{(kr+1)^2-3} \right)|G|+\frac{2kr}{(kr+1)^2-3}$$
unless $G=K_{kr+1}$.
\end{theorem}
\begin{proof}
Let $V$ be the vertex set of $G$, and let $n=|V|$. For a set $X\subseteq V$, let $e(X)$ denote the number of edges of the subgraph $G[X]$ of $G$ induced by $X$. Let $p=kr$ and let
$$R=\left(p+\frac{p-2}{(p+1)^2-3} \right) \mbox{ and } R'=\frac{2p}{(p+1)^2-3}$$
Our aim is to show that $2e(V)\geq Rn+R'$. Let $U=\set{v\in V}{d_G(v)=p}$ be the set of low vertices and let $W=V\sm U$. Note that $d_G(v)\geq p+1$ for all $v\in W$ (by Proposition~\ref{prop:low+vertex}). Note that $p\geq 3r\geq 3$ and $n\geq p+1=kr+1$.
If $U=\ems$, then $2e(V)\geq (p+1)n\geq Rn+R'$ and we are done. So assume that $U\not=\ems$. Let $F=G[U]$ be the low vertex subgraph. If $K=K_{p+1}$ is a subgraph of $F$, then $K$ is a component of $G$. As $G$ has a $\cP$-critical $k$-cover, $G$ is connected. Hence $G=K=K_{kr+1}$ and we are done. So suppose that no subgraph of $F$ is a $K_{p+1}$. Since $p\geq 3r\geq 3$, Theorem~\ref{theorem:lowvertexA1} then implies that $\De(F)\leq p$ and $\De(B)<p$ for all blocks $B\in \cB(F)$. From Theorem~\ref{theorem:Mihok} it then follows that
$$\left( p-1+\frac{2}{p}\right)|U|-2e(U)\geq 2$$
Since every vertex of $U$ has degree $p$ in $G$ and $n=|U|+|W|$, we then obtain that
$$2e(V)=2e(W)+2p|U|-2e(U)\geq 2p|U|-2e(U)\geq \left(p+1-\frac{2}{p}\right)|U|+2$$
On the other hand, since every vertex in $W$ has degree at least $p+1$, we obtain that
$$2e(V)\geq pn+|W|\geq (p+1)n-|U|.$$
Adding the first inequality to the second inequality multiplied with $(p+1-2/p)$ yields
$$2e(V)(p+2-2/p)\geq (p+1-2/p)(p+1)n+2.$$
As $(p+2-2/p)=(p^2+2p-2)/p>0$, this leads to
$$2e(V)\geq \frac{(p^2+p-2)(p+1)n+2p}{p^2+2p-2}=Rn+R'.$$
Thus the proof is complete.
\end{proof}

\begin{corollary}
\label{corollary:Gallai1}
Let $G$ be a simple graph that has an $\cO$-critical $k$-cover of $G$ with $k\geq 3$. Then
$$2|E(G)|\geq \left(k+\frac{k-2}{(k+1)^2-3} \right)|G|+\frac{2k}{(k+1)^2-3}$$
unless $G=K_{k+1}$.
\end{corollary}

For covers associated with constant list assignments Corollary~\ref{corollary:Gallai1} is a reformulation of Gallai's result \cite[Satz 4.4]{Gallai63a} from 1963. For covers associated with general list assignments, Corollary~\ref{corollary:Gallai1} was obtained by Kostochka, Stiebitz, and Wirth \cite{KoStiWi96}. The next corollary for $\cP=\cO$ was obtained by Bernshteyn, Kostochka, and Pron \cite[Corollary 10]{BernsteynKoPro17}.

\begin{corollary}
\label{corollary:Gallai2}
Let $\cP$ be a reliable graph property with $d(\cP)=r$ and let $G$ be a $(\cP,\dcn)$-critical simple graph with $\dcn(G:\cP)=k+1$ and $k\geq 3$. Then
$$2|E(G)|\geq \left(kr+\frac{kr-2}{(kr+1)^2-3} \right)|G|+\frac{2kr}{(kr+1)^2-3}$$
unless $G=K_{kr+1}$.
\end{corollary}

However, the first bound for the number of edges of simple graphs being critical with respect to the chromatic number was obtained in 1957 by Dirac \cite{Dirac57}. In 1974 he proved that his bound is sharp and he characterized the extremal graphs.

For $k\geq 2$, let $\cDG(k)$ denote the family of simple graphs $G$ whose vertex set consists of three non-empty pairwise disjoint sets $A, B_1$ and $B_2$ with
$$|B_1|+|B_2|=|A|+1=k$$
and two additional vertices $v_1$ and $v_2$ such that $A$ and $B_1 \cup B_2$ are cliques in $G$ not joined by any edge, and $N_G(v_i)=A \cup B_i$ for $i=1,2$. Then $G$ has order $2k+1$ and independence number $2$, and so $\cn(G)\geq k+1$. However, if we delete a vertex or an edge, then it is easy to check that the resulting graph has an $\cO$-coloring with $k$ colors. Consequently, if $G\in \cDG(k)$ then $\cn(G-v)<\cn(G)=k+1$ for all $v\in V(G)$ (such graphs are usually called \DF{$(k+1,\cn)$-critical}, similarly we define \DF{$(k+1,\lcn)$-critical} and \DF{$(k+1,\dcn)$-critical}). This implies that if $G\in \cDG(k)$ and $(X,H)$ is the cover of $G$ associated with the constant list assignment $L\equiv [1,k]$, then $(X,H)$ is an $\cO$-critical $k$-cover of $G$. A simple graph $G$ is called \DF{$k$-list-critical} if
$G$ has an $\cO$-critical $k$-cover that is associated with a list assignment $L$, which is the case if and only if $G$ has no $L$-coloring, but $G-v$ has one for all $v\in V(G)$. Every simple graph $G$ that is \DF{$(k+1,\lcn)$-critical} is $k$-list-critical, but not conversely. The standard example is a graph $G$ that is obtained from two disjoint copies of $K_{k+1}$ by adding exactly one edge joining a vertex $u$ of the first copy with a vertex $u'$ of the second copy. The cover $(X,H)$ associated with the list assignment $L$ defined by $L(u)=L(u')=[2,k+1]$ and $L(v)=[1,k]$ is an $\cO$-critical $k$-cover of $G$, and so $G$ is $k$-list-critical, but $G$ is not \DF{$(k+1,\lcn)$-critical} as $\lcn(K_{k+1})=\lcn(G)=k+1$.

In 1957 Dirac proved that every $(k+1,\cn)$-critical graph $G$ distinct from $K_{k+1}$ and with $k\geq 3$ satisfies
$$2|E(G)|\geq k|G|+k-2$$
and in 1974 he proved that equality holds if and only if $G\in \cDG(k)$. In 2002 Kostochka and Stiebitz \cite{KostS2002} proved that every $k$-list-critical graph $G$ not containing $K_{k+1}$ and with $k\geq 3$ satisfies the Dirac bound, and they asked whether equality holds if and only if $G$ belongs to $\cDG(k)$. That this is indeed the case was proved in 2018 by Bernsteyn and Kostochka \cite{BeKo18} by proving the following result.

\begin{theorem}
\label{theorem:Dirac}
Let $G$ be a simple graph that has an $\cO$-critical $k$-cover with $k\geq 3$. If $G$ does not contain $K_{k+1}$ as a subgraph, then
$$2|E(G)|\geq k|G|+k-2$$
and equality holds if and only if $G\in \cDG(k)$.
\end{theorem}

The graphs belonging to $\cDG(k)$ have another interesting feature. As observed by Stiebitz, Tuza, and Voigt \cite{StiebitzTV2013}, if $G\in \cDG(k)$ and $(X,H)$ is a $k$-cover associated with a list assignment $L$ of $G$, then $G$ has no $(\cO,(X,H))$-coloring if and only if $L\equiv [1,k]$ is the constant list assignment. Whether this also holds for arbitrary $k$-covers of $G$ seems to be unknown.

For simple graphs whose order is large, the Gallai bound beats the Dirac bound, however, only if the order is at least quadratic in $k$. Let $f_k(n)$ denote the minimum number of edges in any $(k+1,\cn)$-critical simple graph of order $n$. By K\"onig's theorem, characterizing bipartite graphs (i.e., graphs with $\cn\leq 2$), the only $(3,\cn)$-critical graphs are the odd cycles. So the function is only interesting for $k\geq 3$. For the many partial results obtained for this function the reader is referred to the paper by Kostochka and Yancey \cite{KostY12b} from 2014. Kostochka and Yancey succeeded to determine the best linear approximation for the function $f_k(n)$ with $k\geq 3$, a as consequence they obtained that
$$\lim_{n\to \infty} \frac{2f_k(n)}{n}=k+1-\frac{2}{k}$$
Let $f^\ell_k(n)$ denote the minimum number of edges in any $(k+1,\lcn)$-critical of order $n$, and
let $f^{dp}_k(n)$ denote the minimum number of edges in any $(k+1,\dcn)$-critical simple graph of order $n$.
For both functions we have the Gallai bound as well as the the Dirac bound. For the function $f^{dp}_k(n)$ this seems to be all what is known. For the function $f^\ell_k(n)$ some improvements have been made by Kostochka and Stiebitz \cite{KostStieb03} and more recently by Kierstead and Rabern \cite{KiersteadR2020}. It would be interesting to find further improvements, and to prove or disprove that $f^\ell_k(n)\geq f_k(n)$ ($n\geq k+2\geq 5$).

Given a reliable graph property $\cP$ with $d(\cP)=r$, we say that a graph $G$ is \DF{$(k+1,\cP,\cn)$-critical} if
$\cn(G-v:\cP)<\cn(G:\cP)=k+1$ for all $v\in V(G)$. Let $F_{\cP}(k,n)$ denote the minimum number of edges in any $(k+1,\cP,\cn)$-critical simple graph of order $n$. From Theorem~\ref{theorem:Gallai} it follows that
$$2F_{\cP}(k,n)\geq \left(kr+\frac{kr-2}{(kr+1)^2-3} \right)n+\frac{2kr}{(kr+1)^2-3}.$$
Until now this Gallai type bound is all what is known. One question is whether a Dirac type bound can be proved, at least for some specific properties $\cP$. Apart from the property $\cO$, the best investigated property is $\cD_1$. The class $\cD_d$ of $d$-degenerate (simple) graphs was introduced and investigated in 1970 by Lick and White \cite{LickWhite1970}. For the parameter $\cn(G:\cD_d)$ Lick and White used the term \DF{point partition number} while Bollob\'as and Manvel \cite{BollobasMa79} used the term \DF{$d$-chromatic number}. The point partition number were investigated by various researchers including Lick and White \cite{LickWhite1970},  Kronk and Mitchem \cite{KronkM75}, Mitchem \cite{Mitchem77}, Borodin \cite{Borodin76},  Bollob\'as and Manvel \cite{BollobasMa79}, and possibly others. The term $\cP$-chromatic number was introduced by Hedetniemi \cite{Hedetniemi68} in 1968. He studied, in particular, the $\cD_1$-chromatic number under the name \DF{point aboricity} and proved that any planar graph $G$ satisfies $\cn(G:\cD_1)\leq 3$. Clearly, this is a simple consequence of the fact that any planar graph $G$ is $5$-degenerate; hence we have $\dcn(G:\cD_1)\leq 3$. Note that $\CR(D_d)$ contains all connected $(d+1)$-regular graphs and so $d(\cD_d)=d+1$. This implies, in particular, that
$$2F_{\cD_1}(k,n)\geq \left(2k+\frac{2k-2}{(2k+1)^2-3} \right)n+\frac{4k}{(2k+1)^2-3}.$$
In 2002 \v{S}krekovski \cite{Skrekovski2002} proved that $2F_{\cD_1}(k,n)\geq 2kn+2k-2$, but it is not known whether $F_{\cD_1}(k,n)\geq 2f_k(n)$, provided that $n$ is large enough. The only reliable property $\cP$ for which a Dirac-type bound for the function $F_{\cP}$ is known are the properties $\cD_0=\cO$ and $\cD_1$.

Readers who are interested in additional information concerning the generalized coloring problem are referred to the survey by Albertson, Jamison, Hedetniemi, and Locke \cite{AlbertsonJHL89} and to the survey by Borowiecki and Mih\'ok \cite{BorowMihok91}

\section{Concluding remarks}

Partitioning and coloring graphs under given degree constraints is a well-established area within graph theory and has attracted a lot of attention to date. One of the earliest results in this area was obtained by Lov\'asz \cite{Lovasz66} in 1966. He proved that every simple graph $G$ with $\De(G)<d_1+d_2+ \cdots + d_p$ for $d_1, d_2, \ldots, d_p\in \nat$ has a coloring $\f$ with color set $C=[1,p]$ such that $\De(G[\fin(i)])<d_i$ for all colors $i\in C$. Partitioning of simple graphs into a fixed number of induced subgraphs with bounded degeneracy (coloring number) were first studied in the late 1970s by Borodin \cite{Borodin76} as well as by Bollob\'as and Manvel \cite{BollobasMa79}. Colorings of simple graphs under variable degeneracy constraints were first studied in 2000 by Borodin, Kostochka, and Toft  \cite{BorodinKT00}.  They investigated the following coloring problem for the class of simple graphs; for the class of graphs and hypergraphs this problem was studied by Schweser and Stiebitz \cite{SchweserS2021}.
Let $p\in \nat$ be a fixed integer, and let $(G,\vf)$ be a pair such that $G$ is a graph and $\vf=(f_1,f_2, \ldots, f_p)$ is a vector function of $G$, i.e., $f_i:V(G)\to \nat_0$. We say that $(G,\vf)$ is \DF{colorable} if there is a coloring $\f$ of $G$ with color set $C=[1,p]$ such that $G[\fin(i)]$ is strictly $f_i$-degenerate for all colors $i\in C$, for otherwise we say that $(G,\vf)$ is \DF{uncolorable}. This coloring problem has several interesting applications (see \cite{BorodinKT00}, \cite{Schweser18b} and \cite{SchweserS2021}); the two most popular applications are the following.
If $f_i\equiv 1$ for all $i\in C$, then $(G,\vf)$ is colorable if and only if $\cn(G)\leq p$. Let $L$ be a list assignment with color set $C$, and define, for a vertex $v\in V(G)$ and a color $i\in C$, $f_i(v)=1$ if $i\in L(v)$ else $f_i(v)=0$. Recall that if a subgraph $H$ of $G$ is strictly $f_i$-degenerate, then $V(H)\subseteq \su(f_i)$ implying that $i\in L(v)$ for all $v\in V(H)$. Hence $(G,\vf)$ is colorable if and only if $G$ has a \DF{proper $L$-coloring}, i.e., an $(\cO,L)$-coloring.
Hence the decision problem whether $(G,\vf)$ is colorable is {\NPC}. However, if we add a certain degree condition, this problem can be solved in polynomial time. We call $(G,f)$ \DF{degree-feasible} if every vertex $v\in V(G)$ satisfies
$$\sum_{i=1}^p f_i(v)\geq d_G(v).$$
A good characterization for uncolorable degree feasible pairs $(G,f)$ whose underlying graph $G$ is connected were obtained in \cite{BorodinKT00}, for the class of simple graphs, and in \cite{SchweserS2021}, for the class of graphs and hypergraphs. This characterization can be easily deduced from Theorem~\ref{main_theorem}. To this end, we associate to the pair
$(G,f)$ a configuration $\cK$ as follows:
the fundamental graph of $\cK$ is $G$, the cover of $\cK$ is the cover $(X,H)$ associated to the constant list assignment $L\equiv C=[1,p]$, that is, $X_v=\{v\}\times C$ for all $v\in V(G)$ and for two distinct vertices $(u,i)$ and $(v,j)$ of $H$ we have
$$\mu_H((u,i),(v,j))=
\left\{ \begin{array}{ll}
\mu_G(u,v) & \mbox{\rm if } i=j,\\
0 & \mbox{\rm if } i\not=j,
\end{array}
\right.$$
and the function of $\cK$ is the function $f$ with $f(u,i)=f _i(u)$ for $u\in V(G)$ and $i\in C$.
Then it is easy to check that $(G,f)$ is degree feasible if and only if $\cK=(G,X,H,f)$ is degree feasible; and $(G,f)$ is colorable if and only if $\cK$ is colorable. Hence Theorem~\ref{main_theorem}, respectively Proposition~\ref{prop_construtible-blocks}, yields a constructive characterization for an uncolorable degree-feasible pair $(G,\vf)$, provided that $G$ is a connected graph. This is exactly the characterization given in \cite{BorodinKT00} for simple graphs and in \cite{SchweserS2021} for graphs in general. If $(G,f)$ is an uncolorable degree-feasible pair and $G$ is a block, then it follows from Theorem~\ref{main_theorem} that $(G,f)$ satisfies one of the following three conditions:
\begin{itemize}
  \item There is an integer $j$ such that $f_j(v)=d_G(v)$ and $f_i(v)=0$ for $i\not=j$ and $v\in V(G)$.
  \item $G=tK_n$ for some integers $t,n\in \nat$ and there are integers $n_1,n_2, \ldots, n_p\in \nat_0$ such that $n_1+n_2 + \cdots +n_d=n-1$ and $\vf(v)=(tn_1,tn_2, \ldots, tn_p)$ for all $v\in V(G)$.
  \item $G=tC_n$ for some integers $t,n$, where $t\geq 1$ and $n\geq 3$ is odd, and there are two integers $k, \ell \in [1,p]$ such that
    $$f_i(v)=
\left\{ \begin{array}{ll}
t & \mbox{\rm if } i\in \{k,\ell\},\\
0 & \mbox{\rm if } i\in [1,p]\sm \{k,\ell\}
\end{array}
\right.$$
for all $v\in V(G)$.
\end{itemize}
Note that if $G$ is a block, the configuration associated to $(G,f)$ can never be an even C-configuration. Consequently, Theorem~\ref{main_theorem} is a far reaching generalization of many well known and interesting results related to ordinary colorings as well as to generalized colorings of graphs. That it is worthwhile to study these coloring problems also for graphs having multiple edges was first pointed out by Kim and Ozeki \cite{KimOz17}; they used these concepts to study colorings of signed graphs. As demonstrated by the second author in his thesis (Coloring of Graphs, Digraphs, and Hypergraphs, TU Ilmenau, 2020) the characterization of uncolorable pairs for graphs in general can be used to obtain Brooks type results for the dichromatic number and list dicromatic number of digraphs; such results were first obtained by Harutyunyan and Mohar \cite{HarutMo2012} in 2012. The decomposition result of Laslo Lov\'asz has a short and elegant proof. Moreover, it has motivated a large number of follow-up investigations in this direction. Two more recent papers about partitioning and coloring graphs with degree constrains were published by Landon Rabern, see \cite{Rabern2012b} and \cite{Rabern2013b}.

\end{document}